\documentclass{article}
\usepackage{babel}
\usepackage[utf8]{inputenc}
\usepackage[T1]{fontenc}
\usepackage[top=2.5cm, bottom=2.5cm, left=2.5cm, right=2.5cm]{geometry} 
\usepackage{verbatim}
\usepackage{amsmath}
\usepackage{amsthm}
\usepackage{amsfonts}
\usepackage{mathabx}
\usepackage{color}
\usepackage{xcolor}
\usepackage{amssymb}
\usepackage{mathrsfs}
\usepackage{subfigure}
\usepackage{float}
\usepackage[breaklinks=true]{hyperref}
\usepackage{xspace}
\usepackage{enumitem}
\usepackage{makecell}
\usepackage{algorithm}
\usepackage{algpseudocode}
\usepackage[onehalfspacing]{setspace}

\renewcommand{\H}{\mathcal{H}}

\newcommand{\syst}[1]{\left \{ \begin{array}{l} #1 \end{array} \right. \kern-\nulldelimiterspace}	

\newcommand{\argmin}{\text{\normalfont argmin}}

\newcommand{\dom}{\text{\normalfont dom}\,}

\renewcommand{\int}{\text{\normalfont int}\,}

\makeatletter
\newlength{\algorithmboxrule}
\makeatother

\newtheorem{proposition}{Proposition}
\newtheorem{corollary}{Corollary}
\newtheorem{lemma}{Lemma}
\newtheorem{theorem}{Theorem}
\newtheorem{definition}{Definition}
\newtheorem{remark}{Remark}
\newtheorem{example}{Example}

\def\qed{\hbox to 0pt{}\hfill$\rlap{$\sqcap$}\sqcup$}
\makeatother

\title{Proximity operator characterization for abstract convex functions}
\author{ Ewa Bednarczuk\thanks{Warsaw University of Technology,  00-662 Warsaw, Koszykowa 75, Poland} \thanks{Systems Research Institute, PAS,  01-447 Warsaw, Newelska 6, Poland}
\and
The Hung Tran\footnotemark[2]}
\date{}

\usepackage{graphicx}

\begin{document}

\maketitle
\begin{center}
In honour of Professor Phan Quoc Khanh
\end{center}
\begin{abstract}
We consider proximity operator as a selector of the subgradient in the context of abstract convexity and characterize its properties in term of minimization problems. We also investigate the continuity and differentiability of proximity operator in the case of paraconvex and weakly convex functions. Supporting examples are given in each section.

\end{abstract}
\section{Introduction}
Proximal operator was introduced by Moreau \cite{moreau1965proximite,moreau1970} and became one of the most important concept in modern optimization and applications. 
Since then, numerous works have been devoted to theoretical and numerical aspects of proximal operator \cite{jourani2014differential,bauschke2018regularizing,chambolle2011first}. The strength of proximity operator lies in its nonexpansivity and the differentiability of the corresponding Moreau envelope. 
Its characterizations have been discussed in  \cite{gribonval2020characterization}. 
Recently, proximal operator and Moreau envelope proved their efficiency when applying to weakly convex functions \cite{hoheisel2010proximal,davis2019}. 
Other directions of research including new type of proximity operator such as Bregman proximal operator \cite{kiwiel1997proximal} or proximity operator for Legendre type functions \cite{laude2023anisotropic}, higher order Moreau envelope \cite{kabgani2025moreau}.

The starting point of our analysis comes from Moreau \cite[Corollary 10.c]{moreau1965proximite} which states a single-valued operator $T$ acting on Hilbert space $H$ is a proximity operator if and only if $T$ is nonexpansive on $H$ and there exists a convex function $\psi:H\to (-\infty,+\infty]$ such that $T(x)\in \partial \psi (x)$ for any $x\in H$. Moreover, $T(x)$ is a unique solution of the minimization problem $\min_{z\in H} \varphi(z)+\frac{1}{2}\Vert z-x\Vert^2$ where $\varphi\in \Gamma_0(H)$. This result has been generalized in \cite{gribonval2020characterization} by admitting penalty terms other than $\frac{1}{2}\Vert z-\cdot\Vert^2$.

In the present paper, we further generalize the results of \cite{gribonval2020characterization}, by admitting nonconvex $\varphi$ which satisfy abstract convexity conditions, the term coined by Rubinov \cite{Rub2013}, and Pallaschke and Rolewicz \cite{Pall2013}. 
We show that our results apply to the concept of proximity operator, namely  $\Phi_{lsc}^\mathbb{R}$-proximity operator (Corollary \ref{cor: lsc prox relation}) where $\Phi_{lsc}^\mathbb{R}$ is the class of quadratic functions \cite{bednarczuk2025proximal}. Furthermore, we specify Theorem \ref{thm: prox relation general} for weakly convex functions and paraconvex functions by using proximal subdifferentials (Corollaries \ref{cor: lsc prox relation general} and \ref{cor: paraconvex prox relation general}). 
We also show that the continuity of the proximity operator is related to the differentiability of the functions $\varphi$ and $\psi$ (see Corollary \ref{cor: smooth prox paraconvex}) through the continuous selection of suitably chosen subdifferentials. 

The organization of the paper is as follows: we present the basic definitions and notation in Section \ref{sec: basic def and note}. Section \ref{sec: prox main result} contains our main results on the proximity operator within abstract convexity setting (Theorem \ref{thm: prox relation general}). 
We discuss the continuity of the proximity operator in section \ref{sec: smooth prox}. Section \ref{sec: supporting proof} contains auxiliary results and some relevant properties of paraconvex and weakly convex functions. We close the paper with the conclusion.

\section{Preliminaries and Notation}
\label{sec: basic def and note}

Let $H$ be a Hilbert space and let $G$ be a linear space. For a function $\varphi:H \to (-\infty,+\infty]$, $\mathrm{dom}\ \varphi:=\{ x\in H: \varphi(x) <+\infty\}$ and $\mathrm{Im}\ \varphi :=\{y\in \mathbb{R}: \exists x\in H, f(x)=y\}$ are the domain and range of $\varphi$, respectively.
A function is $\varphi$ is proper if $\mathrm{dom }\varphi\neq\emptyset$.
For a set-valued operator $A: H \rightrightarrows G$, its domain and range are respectively defined as
\begin{equation}
    \dom A = \left\{ x\in H: Ax \neq \emptyset\right\},\quad \text{ran } A = \left\{ Ax: x\in H\right\}.
\end{equation}
We define its inverse $A^{-1}:G\rightrightarrows H, A^{-1}(g) = \{ x\in H: g\in A(x)\}$ for $g\in G$. It holds $\dom A^{-1} = \text{ran } A$.

Let $\Phi$ be the collection of functions $\Phi:=\{\phi:H\to \mathbb{R}\}\subset G$ which is closed under addition of constant and we set $\langle \phi, x\rangle_\Phi = \phi(x)$ for $x\in H,\phi \in \Phi$.
For instance, the class $\Phi$ can be taken as the class of affine functions $\Phi_{\mathrm{affine}}=\{\phi=(u,c)\in H\times \mathbb{R}: \phi(x) =\langle u,x\rangle+c, x\in H\}$ or quadratic $\Phi_{lsc}^\mathbb{R}=\{\phi=(a,u,c)\in \mathbb{R}\times H\times \mathbb{R}: \phi(x) = -a\Vert x\Vert^2+\langle u,x\rangle+c\}$. Clearly, $\Phi_{lsc}^\mathbb{R}$ is a Hilbert space as a Cartesian product of Hilbert spaces.

We introduce the notion of $\Phi$-convexity.
\begin{definition}\cite{Pall2013,Rub2013}
\label{def:Phi-convex}
For a given $x\in H$, a function $\varphi:H\to\left(-\infty,+\infty\right]$ is said to be $\Phi$-convex at $x$
if and only if 
\[
\varphi\left(x\right)= \varphi^\Phi (x),
\]
where $\varphi^\Phi (x) = \sup_{\phi\in\Phi,\phi\leq \varphi} \phi\left(x\right)$ is the $\Phi$-convexification of $\varphi$ at $x\in H$. If the above relation holds for any $x\in H$, we say that $\varphi$ is $\Phi$-convex on $H$.
\end{definition}
Observe that the above definition is independent of the structure of $H$, which means that Definition remains meaningful if instead of Hilbert space $H$, we consider functions $\phi$ define on a set $X$, see also \cite[Page 18]{Pall2013}.
Together with abstract convex functions, we define paraconvex functions.
\begin{definition}
\label{def:paraconvex}
Let $2\geq\gamma>1$, we say that $\varphi:H\to (-\infty,+\infty]$ is $\gamma$-paraconvex if there exists a constant $C>0$ such that for any $x,y\in H,\lambda\in [0,1]$, we have
\[
\varphi\left(\lambda x+ (1-\lambda)y \right) \leq \lambda \varphi(x) +(1-\lambda) \varphi (y) +C\lambda(1-\lambda)\Vert x-y\Vert^\gamma.
\]
\end{definition}
When $\gamma=2$, we recover the definition of weakly convex functions, equivalently $\varphi+\rho\Vert\cdot\Vert^2$ is a convex function. When $\gamma<2$, this equivalence fails to hold, for example, see \cite[Example 1]{rolewicz1979gamma}.

Together with Definition \ref{def:Phi-convex} of $\Phi$-convex functions, we introduce several concepts of subgradients. 
\begin{definition}
Let $\varphi:H\to (-\infty,+\infty]$ be proper and $x\in \mathrm{dom}\  \varphi$.
We say that $\phi\in \Phi$ is a $\Phi$-subgradient of $\varphi $ at $x$ if
    \[
    (\forall y\in H)\quad \varphi(y) -\varphi(x) \geq \phi(y)-\phi(x).
    \]
    We denote by $\partial_\Phi \varphi(x)$ the collection of all $\Phi$-subgradient of $\varphi$ at $x$.
    
On the other hand, for a proper function $\psi:\Phi\to (-\infty,+\infty]$ and $\phi\in \mathrm{dom }\psi$, we say that $x\in H$ is $\Phi$-subgradient of $\psi$ at $\phi$ if and only if
\[
(\forall \phi'\in\Phi) \quad \psi (\phi') - \psi(\phi) \geq \phi'(x)-\phi(x).
\]
The collection of all $\Phi$-subgradients at $\psi(\phi)$ is denoted by $\partial_\Phi \psi(\phi) \subset H$.
\end{definition}

\begin{proposition}[Proposition 1.2 \cite{Rub2013}]
\label{prop: phi-convex subgradient}
Let $\varphi:H\to (-\infty,+\infty]$ be proper and $y\in \mathrm{dom }\varphi$, then the $\Phi$-subdifferential $\partial_\Phi\varphi (y)$ is nonempty if and only if
\[
\varphi(y) = \max \{\phi(y): \phi\leq \varphi, \phi\in\Phi\}.
\]
\end{proposition}
By Definition \ref{def:Phi-convex}, the above equality means that $\varphi(y) = \varphi^\Phi (y)$, i.e. $\varphi$ is $\Phi$-convex at $y\in \mathrm{dom }\varphi$.
In the original version, Proposition \ref{prop: phi-convex subgradient} has been formulated for functions $\phi$ defined on any set $X$. This allows us to obtain the following proposition as well.
\begin{proposition}
\label{prop: phi-convex phi-subgradient}
Let $\psi:\Phi\to (-\infty,+\infty]$ be proper and $\phi\in \mathrm{dom }\psi$, then the $H$-subdifferential $\partial_\Phi\psi (\phi)$ is nonempty if and only if
\[
\psi(\phi) = \max \{y(\phi): y(\phi)\leq y(\psi), \forall y\in H\}.
\]
\end{proposition}
By Definition \ref{def:Phi-convex} and the remark below, the above inequality means that $\psi$ is $H$-convex.

\begin{definition}
Let $\varphi:H\to (-\infty,+\infty]$ be proper and $x\in \mathrm{dom}\  \varphi$.
\begin{itemize}
    \item For $1< \gamma \leq 2$, an element $u\in H$ is a $\gamma$-subgradient of $\varphi $ at $x$ if there exists a constant $C>0$ such that
    \[
    (\forall y\in H)\quad \varphi(y) -\varphi(x) \geq \langle u,y-x\rangle -C\Vert x-y\Vert^\gamma.
    \]
    We denote by $\partial_{\gamma} \varphi(x)$ the collection of all $\gamma$-subgradient of $\varphi$ at $x$.
    When the above inequality holds locally around a given $x$, the set of all local $\gamma$-subgradients of $\varphi$ at $x$ is denoted by $\partial_{\gamma,\mathrm{loc}} \varphi(x)$.
    \item When $\gamma=2$, and  $C=\rho/2>0$ is fixed, an element $u\in H$ satisfying the above inequality is a $\rho$-proximal subgradient of $\varphi$ at $x$ and we denote the collection of all $\rho$-proximal subgradients as $\partial_{(2,\rho)} \varphi (x)$.
    \item Let $x\in V\subset H, 2\geq \gamma>1$, we define the $\gamma$-proximal normal cone (see \cite{jourani1996subdifferentiability} for related definitions) as 
    \[
    N_{\gamma} (x;V):=\{ v\in H: \exists C>0, \forall y\in V, \quad \langle v,y-x\rangle \leq C\Vert x-y\Vert^\gamma\}.
    \]
    When $\gamma=2$, we recover the proximal normal cone. (see \cite[Chapter 1.1]{clarke2008nonsmooth})
\end{itemize}
\end{definition}

\begin{remark}
\label{rmk: weak phi subdiff}
The definition of $\gamma$-subdifferentials coincides with the definition of $\Phi$-weak subdifferentials given by \cite[Formula 2]{rolewicz2000cyclic} for $\alpha(t) = C\Vert t\Vert^\gamma$.
\end{remark}

\begin{remark}
\label{rmk: gamma subgrad gateaux}
It has been proved by \cite{jourani1996subdifferentiability} and \cite{giles1999survey} that the $\gamma$-subdifferentials at a point $x\in H$ coincides with G\^ateaux derivative $\nabla \varphi(x)$ whenever $\varphi$ is G\^ateaux differentiable. This property does not hold for $\Phi$-subdifferentials.
\end{remark}

\section{Proximal operator in context of abstract convexity}
\label{sec: prox main result}

The concept of $\Phi$-proximity operator has already appeared in some numerical algorithms \cite{bednarczuk2025proximal}. 
The authors in \cite{andra2002} proposed a cutting plane method using $\Phi$-subgradient, and it has been further improved by \cite{rakotomandimby2026subgradient} to apply to the class of "capra"-convex funtions. Recently, \cite{bednarczuk2026outer} proposes an outer approximation using quadratic cuts based on weakly convex functions.

Inspired by the results of \cite{gribonval2020characterization}, we investigate $\Phi$ proximity operator $f:\Phi\to H$ in the context of abstract convexity.
By using $\Phi$-subdifferentials, we extend the applicability of \cite[Theorem 3]{gribonval2020characterization} to nonconvex problems, namely $\Phi$-convex problems.
We start with the main theorem.

\begin{theorem}
\label{thm: prox relation general}
Let $J:\Phi \to(-\infty,+\infty],I:H\to(-\infty,+\infty],\mathcal{A}:\Phi \to\Phi, \mathcal{B}:H\to H$.
Consider $f:\Phi\to H$, its image set $\mathcal{\mathrm{Im}}f$ and $\mathrm{dom}\ f \subset \mathrm{dom}\ \mathcal{A}$.
\begin{enumerate}
\item Let $D\left(x,\phi\right)=J\left(\phi\right)-\left\langle \mathcal{A}\left(\phi\right),x\right\rangle _{\Phi}+I\left(x\right)$, the following are equivalent: 
\begin{enumerate}[label=\roman*.]
\item There exists $\varphi:H\to(-\infty,+\infty]$ such that $f\left(\phi\right)\in\arg\min_{z\in H}D\left(z,\phi\right)+\varphi\left(z\right)$
for each $\phi\in\Phi$.
\item There is a $\Phi$-convex $g:H\to(-\infty,+\infty]$ on $\mathrm{Im }f \subset\mathrm{dom } g$
such that $\mathcal{A}\left(f^{-1}\left(x\right)\right)\subset\partial_{\Phi}g\left(x\right)$ for each $x\in\mathrm{Im}f$.
\end{enumerate}
\item Let $\tilde{D}\left(x,\phi\right)=J\left(\phi\right)-\left\langle \phi, \mathcal{B}\left(x\right)\right\rangle _{\Phi}+ I\left(x\right)$.
The following are equivalent:
\begin{enumerate}[label=\roman*.]
\item There exists $\varphi:H\to(-\infty,+\infty]$ such that $f\left(\phi\right)\in\arg\min_{x\in H}\tilde{D}\left(x,\phi\right)+\varphi\left(x\right)$
for each $\phi\in\Phi$.
\item There is a $H$-convex $\psi:\Phi \to(-\infty,+\infty]$
such that $\mathcal{B}\left(f\left(\phi\right)\right)\in \partial_{\Phi}\psi\left(\phi\right)$ for each $\phi\in\Phi$.
\end{enumerate}
\end{enumerate}
\end{theorem}
\begin{proof}
For 1-(i) $\Rightarrow$ 1-(ii), let us set $\theta:H \to (-\infty,+\infty], \theta := I+\varphi+\iota_{\mathrm{Im}f}$
and let $x=f\left(\phi\right)$ where $\phi\in\Phi$.
By assumption, $x$ is a minimizer of $D\left(z,\phi\right)+\varphi\left(z\right)$
i.e.
\begin{align*}
\left(\forall z\in H\right)\qquad D\left(z,\phi\right)+\varphi\left(z\right) & \geq D\left(x,\phi\right)+\varphi\left(x\right)\\
-\left\langle \mathcal{A}\left(\phi\right),z\right\rangle_\Phi +I\left(z\right)+\varphi\left(z\right) & \geq-\left\langle \mathcal{A}\left(\phi\right),x\right\rangle_\Phi +I\left(x\right)+\varphi\left(x\right).
\end{align*}
Since $x\in \mathrm{Im}\ f$, we have
\[
\left(I+\varphi+\iota_{\mathrm{Im}f}\right)\left(z\right)-\left\langle \mathcal{A}\left(\phi\right),z\right\rangle_\Phi \geq I\left(z\right)+\varphi\left(z\right)-\left\langle \mathcal{A}\left(\phi\right),z\right\rangle_\Phi \geq\left(I+\varphi+\iota_{\mathrm{Im}f}\right)\left(x\right)-\left\langle \mathcal{A}\left(\phi\right),x\right\rangle_\Phi ,
\]
which implies $\mathcal{A}\left(\phi\right)\in\partial_{\Phi}\theta\left(x\right)=\partial_{\Phi}\theta\left(f\left(\phi\right)\right)$
or $\mathcal{A}\left(f^{-1}\left(x\right)\right)\subset\partial_{\Phi}\theta\left(x\right)$
as $x=f\left(\phi\right)$. Hence, 1-(ii) holds. To see that $g$ is $\Phi$-convex, it is enough to note that $\partial_\Phi g(x)\neq\emptyset$ for every $x\in \mathrm{Im }f$ and by Proposition \ref{prop: phi-convex subgradient}, $g$ is $\Phi$-convex on $\mathrm{Im }f$.

Conversely, let 1-(ii) holds and $\theta_1:H\to (-\infty,+\infty], \theta_{1}=g+\iota_{\mathrm{Im}f}$.
By assumption, we have $\mathcal{A}\left(f^{-1}\left(x\right)\right)\subset \partial_{\Phi}g\left(x\right)\neq\emptyset$ for any $x\in\mathrm{Im}f$.
Since $\mathrm{dom}\ \partial_{\Phi}g\subset\mathrm{dom}\ g$, we have
$\mathrm{Im}f\subset\mathrm{dom}g$. Then 
\[
\mathrm{dom}\theta_{1}=\mathrm{dom}g\cap \mathrm{Im}f=\mathrm{Im}f.
\]
Let $\phi\in\Phi$ and $x=f\left(\phi\right),$ $x\in\mathrm{Im}f$. We have $\theta_{1}\left(x\right)=g\left(x\right)$
and 
\[
\mathcal{A}\left(\phi\right)\in\mathcal{A}\left(f^{-1}\left(x\right)\right)\subset\partial_{\Phi}g\left(x\right).
\]
By the definition of subdifferentials, we have 
\[
\left(\forall z\in H\right)\qquad\theta_{1}\left(z\right)-\theta_{1}\left(x\right)\geq g\left(z\right)-g\left(x\right)\geq \left\langle \mathcal{A}\left(\phi\right),z\right\rangle_\Phi -\left\langle \mathcal{A}\left(\phi\right),x\right\rangle_\Phi ,
\]
so $\mathcal{A}\left(\phi\right)\in\partial_{\Phi}\theta_{1}\left(x\right)$.
Let $\varphi=\theta_{1}-I$, then $x$ is a minimizer of $\arg\min_{z\in H}D\left(z,\phi\right)+\varphi\left(z\right)$.

For 2-(i)$\Rightarrow$ 2-(ii): Let us fix $x=f(\phi)$ from 2-(i) and for any $\phi' \in \Phi$, set $x'= f(\phi')$ as the minimizer of $\varphi+\tilde{D}(\cdot,\phi')$ which exists by assumption. We have
\begin{align}
& \varphi (x) +\tilde{D}(x,\phi')  \geq \varphi (x') +\tilde{D} (x',\phi') \nonumber\\
\Leftrightarrow \ & \varphi (x) + I(x)- \langle \phi', \mathcal{B}(x)\rangle_\Phi \geq \varphi(x') +I(x') - \langle \phi',\mathcal{B}(x')\rangle_\Phi. \label{eq: prop x_prime proof}
\end{align}

We denote $\psi:\Phi\to (-\infty,+\infty], \psi(\phi) :=-\varphi(f(\phi)) -I(f(\phi)) + \langle \phi, \mathcal{B}(f(\phi))\rangle_\Phi  $. Then \eqref{eq: prop x_prime proof} becomes
\[
-\psi\left(\phi'\right)+\left\langle \phi', \mathcal{B}\left(f\left(\phi\right)\right)\right\rangle _{\Phi} \leq-\psi\left(\phi\right)+\left\langle \phi, \mathcal{B}\left(f\left(\phi\right)\right)\right\rangle _{\Phi}.
\]
As the above inequality holds for any $\phi'\in\Phi$, we conclude that $\mathcal{B}(f(\phi)) \in \partial_{\Phi} \psi(\phi)$ which is 2-(ii). The $H$-convexity of $\psi$ is supported by Proposition \ref{prop: phi-convex phi-subgradient}.
The other relation 2-(ii)$\Rightarrow$(i) can be proved analogously as 1-(ii)$\Rightarrow$1-(i).
\end{proof}

We say that $f:\Phi\to H$ is a \textit{$\Phi$-proximity operator} if it satisfies either one of the two statements in Theorem \ref{thm: prox relation general}-1(i) or (ii), Theorem \ref{thm: prox relation general}-2(i) or (ii).

Let us recall that the classical proximal operator of any function $\varphi:H\to(-\infty,+\infty]$ at $x_0$ with parameter $\lambda>0$ is defined as,
\begin{equation}
\label{eq: classic prox}
f(x_0) = \argmin_{x\in H} \varphi(x) + \frac{1}{2\lambda}\Vert x-x_0\Vert^2.
\end{equation}
Hence, we define $D(x,x_0)= \frac{1}{2\lambda}\Vert x\Vert^2 - \frac{1}{\lambda}\langle x,x_0\rangle +\frac{1}{2\lambda}\Vert x_0\Vert^2$ or equivalently $D(x,\phi) = I(x) -\phi(x)+J(\phi)$ where $\phi(x) = \langle \frac{x_0}{\lambda},x\rangle$ and $J(\phi) = \frac{1}{2\lambda}\Vert x_0\Vert^2, I(x) = \frac{1}{2\lambda}\Vert x\Vert^2 $ or $\phi=x_0$.

\begin{remark}
In Theorem \ref{thm: prox relation general}, we require the existence of global minimizers for each $\phi\in\Phi$. 
Theorem \ref{thm: prox relation general}-1 remains true if we limit $\phi\in O\subset\Phi$.

In Theorem \ref{thm: prox relation general}-2, this is a strong assumption even in simple case e.g. let us take $\varphi (x)=I(x) =x^2$ and $\phi(x) = -a x^2$, then $\varphi(x) + D(x,\phi) = (2+a)x^2$ has a minimizer at $\phi$ with $a>-2$.
To avoid this drawback, one can modify Theorem \ref{thm: prox relation general}-2 by considering the minimizers to $\phi\in \mathcal{O}\subset \Phi$ and to consider local $\H$-subdifferentials.
\end{remark}

For $\Phi=\Phi_{\mathrm{affine}}$, Theorem \ref{thm: prox relation general} reduces to Theorem 3 in \cite{gribonval2020characterization}.
Let us give the definition of polygonally connected set which allows us to be more specific about the form of functions $g$ below. We say that a set $E\subset H$ is polygonally connected if for every pair $a,b \in E$, there exists $x_0=a, x_1,\cdots,x_n =b \in E$ such that $\bigcup_{0\leq i\leq n-1} [x_i,x_{i+1}] \subset E$, where $[a,b]$ is a segment from $a$ to $b$.

\begin{corollary}
\label{cor: lsc prox relation general}
Let $\Phi=\Phi_{lsc}^\mathbb{R}$, and let $J:\Phi_{lsc}^{\mathbb{R}}\to(-\infty,+\infty],I:H\to(-\infty,+\infty],\mathcal{A}:\Phi_{lsc}^{\mathbb{R}}\to\Phi_{lsc}^{\mathbb{R}},\mathcal{B}:H\to H$.
Consider $f:\Phi_{lsc}^{\mathbb{R}}\to H$ and assume $\mathrm{dom}\ f \subset\mathrm{dom}\ \mathcal{A}$.
\begin{enumerate}
\item Let $D\left(x,\phi\right)=J\left(\phi\right)- \langle\mathcal{A}\left(\phi\right), x\rangle_\Phi +I\left(x\right)$ for $\phi\in\Phi_{lsc}^\mathbb{R}$,
the following are equivalent: 
\begin{enumerate}[label=\roman*.]
\item There exists $\varphi:H\to(-\infty,+\infty]$ such that $f\left(\phi\right)\in\arg\min_{z\in H}D\left(z,\phi\right)+\varphi\left(z\right)$
for each $\phi\in\Phi_{lsc}^{\mathbb{R}}$.
\item There is a $\Phi_{lsc}^{\mathbb{R}}$-convex $g:H\to(-\infty,+\infty]$
such that $\mathcal{A}\left(f^{-1}\left(x\right)\right)\subset\partial_{\Phi}g\left(x\right)$
for each $x\in\mathrm{Im}f$.
\end{enumerate}
\item Let $\varphi$ and $g$ satisfy the above equivalence and let $E\subset\mathrm{Im}f$
be polygonally connected, then there exists $K\in\mathbb{R}$ such that 
\[
g\left(x\right)=I\left(x\right)+\varphi\left(x\right)+K,\quad x\in E.
\]
\item Let $\tilde{D}\left(x,\phi\right)=J\left(\phi\right)-\left\langle \phi, \mathcal{B}\left(x\right)\right\rangle _{\Phi}+ I\left(x\right)$.
The following are equivalent:
\begin{enumerate}[label=\roman*.]
\item There exists $\varphi:H\to(-\infty,+\infty]$ such that $f\left(\phi\right)\in\arg\min_{x\in H}\tilde{D}\left(x,\phi\right)+\varphi\left(x\right)$
for each $\phi\in\Phi_{lsc}^{\mathbb{R}}$.
\item There is a $\Phi_{lsc}^{\mathbb{R}}$-convex $\psi:\Phi_{lsc}^\mathbb{R} \to(-\infty,+\infty]$
such that $\mathcal{B}\left(f\left(\phi\right)\right)\in \partial_{\Phi}\psi\left(\phi\right)$
for each $\phi\in\Phi_{lsc}^{\mathbb{R}}$.
\end{enumerate}
\item If $\varphi$ and $\psi$ satisfy the third equivalence, and let $E'\subset \Phi_{lsc}^\mathbb{R}$ be polygonally connected. Then there exists a constant $G\in\mathbb{R}$ such that
\[
\psi(\phi) = \langle \phi, \mathcal{B} (f(\phi))\rangle_\Phi - J(\phi) - \varphi (f(\phi)) +G, \quad \forall \phi \in E'.
\]
\end{enumerate}
\end{corollary}
\begin{proof}
We only prove the second statement, the forth statement can be proved similarly. 
Consider the functions $g$ and $\varphi$ satisfying Theorem \ref{thm: prox relation general}-1(i,ii). We obtain that $\mathcal{A}(f^{-1}(x)) \subset \partial_\Phi g(x) \cap \partial_\Phi (\varphi +I)(x) \neq\emptyset$. Applying Lemma \ref{lem: lsc conv poly-connected} (see Section \ref{sec: supporting proof} below) with $E$ polygonally connected, there exists a constant $K\in\mathbb{R}$ such that for any $x\in E$,
\[
g(x) = I(x)+\varphi (x) +K.
\]
\end{proof}
Corollary \ref{cor: lsc prox relation general} holds true for functions $g$ and $\psi$ which are weakly convex and $\gamma$-paraconvex functions ($2\geq\gamma>1)$ by using proximal subdifferentials or $\gamma$-subdifferentials, respectively.
Before investigating this, we observe that $\gamma$-subdifferentials are $\gamma$-monotone in the sense that there exists a constant $C>0$ such that
\[
\langle x_1-x_2,u_2-u_1\rangle \leq C\Vert x_1-x_2\Vert^\gamma,
\]
for all $u_1 \in \partial_{\gamma}\varphi (x_1)$ and $u_2 \in \partial_{\gamma}\varphi (x_2)$, where $x_1,x_2\in\mathrm{dom }\varphi$.
The following fact holds.
\begin{theorem}[Theorem 7.1 \cite{jourani1996subdifferentiability}]
\label{thm: gamma monotone}
Let $H$ be Hilbert and $\varphi:H\to (-\infty,+\infty]$ be a proper lower-semicontinuous function. Suppose that $\partial_{\gamma} \varphi (x)$ is nonempty for $x\in \mathrm{dom }\varphi$, then $\varphi$ is $\gamma$-paraconvex with constant $C$.
\end{theorem}


\begin{corollary}
\label{cor: paraconvex prox relation general}
Let $H$ be a Hilbert space and $I,J:H\to(-\infty,+\infty]$ are continuous, $\mathcal{A}:\H \to H,\mathcal{B}:H\to H$.
Consider $1<\gamma\leq 2$, $f:H\to H$, assume $\mathrm{Im }f$ is a closed set.
\begin{enumerate}
\item Let $D\left(x,y\right)=J\left(y \right)- \langle\mathcal{A}\left(y \right), x\rangle +I\left(x\right)$ for $y\in H$,
the following are equivalent: 
\begin{enumerate}[label=\roman*.]
\item There exists a lower semicontinuous $\varphi:H\to(-\infty,+\infty]$ such that $f\left(y\right)\in\arg\min_{x\in H}D\left(x,y\right)+\varphi\left(x\right)$
for each $y\in H$.
\item There is a lower semicontinuous $\gamma$-paraconvex $g:H\to(-\infty,+\infty]$
such that $\mathcal{A}\left(f^{-1}\left(x\right)\right)\subset\partial_{\gamma}g\left(x\right)$
for each $x\in\mathrm{Im}f$.
\end{enumerate}
\item Let $\varphi$ and $g$ satisfy the above statement and let $E\subset\mathrm{Im}f$
be polygonally connected, then there exists $K\in\mathbb{R}$ such that 
\[
g\left(x\right)=I\left(x\right)+\varphi\left(x\right)+K,\quad x\in E.
\]
\item Let $\tilde{D}\left(x,y\right)=J\left(y\right)-\left\langle y, \mathcal{B}\left(x\right)\right\rangle+ I\left(x\right)$.
The following are equivalent:
\begin{enumerate}[label=\roman*.]
\item There exists a lower semicontinuous $\varphi:H\to(-\infty,+\infty]$ such that $f\left(y\right)\in\arg\min_{x\in H}\tilde{D}\left(x,y\right)+\varphi\left(x\right)$
for each $y\in H$.
\item There is a lower semicontinuous $\gamma$-paraconvex $\psi:H \to(-\infty,+\infty]$
such that $\mathcal{B}\left(f\left(y\right)\right)\in \partial_{\gamma}\psi\left(y\right)$
for each $y\in H$.
\end{enumerate}
\item If $\varphi$ and $\psi$ satisfy the third equivalence, and let $E'\subset H$ be polygonally connected. Then there exists a constant $G\in\mathbb{R}$ such that
\[
\psi(y) = \langle y, \mathcal{B} (f(y))\rangle - J(y) - \varphi (f(y)) +G, \quad \forall y \in E'.
\]
\end{enumerate}
\end{corollary}
\begin{proof}
The proof follows the lines of the proof of Theorem \ref{thm: prox relation general}. We only present the differences.
For 1-(i) $\Rightarrow$ 1-(ii), let us set $\theta:H \to (-\infty,+\infty], \theta := I+\varphi+\iota_{\mathrm{Im}f}$
and let $x=f\left(y\right), y\in H$.
By assumption, $x$ is a minimizer of $D\left(z,y\right)+\varphi\left(z\right)$
i.e.
\begin{align*}
\left(\forall z\in H\right)\qquad D\left(z,y\right)+\varphi\left(z\right) & \geq D\left(x,y\right)+\varphi\left(x\right)\\
I\left(z\right)+\varphi\left(z\right) & \geq \left\langle \mathcal{A}\left(y\right),z-x\right\rangle +I\left(x\right)+\varphi\left(x\right)\\
& \geq \left\langle \mathcal{A}\left(y\right),z-x\right\rangle +I\left(x\right)+\varphi\left(x\right) -C\Vert z-x\Vert^\gamma,
\end{align*}
for some constant $C>0$.
We deduce that $\mathcal{A}\left(y\right)\in\partial_{\gamma}\theta\left(x\right)=\partial_{\gamma}\theta\left(f\left(y\right)\right)$
or $\mathcal{A}\left(f^{-1}\left(x\right)\right)\subset\partial_{\gamma}\theta\left(x\right)$
as $x=f\left(y\right)$. Hence, 1-(ii) holds with $g=\theta$. By assumption, we see that $g$ is lower semicontinuous and $\partial_\gamma g(x)\neq\emptyset$, so it is $\gamma$-paraconvex by Theorem \ref{thm: gamma monotone}.

Conversely, let 1-(ii) holds and set $\theta_{1}=g+\iota_{\mathrm{Im}f}$ which is lower semicontinuous.
By assumption, we have $\mathcal{A}\left(f^{-1}\left(x\right)\right)\subset \partial_{\gamma}g\left(x\right)\neq\emptyset$ for any $x\in\mathrm{Im}f$.
Let us fix an arbitrary $y\in H$ and set $x=f\left(y\right)$. We have $\theta_{1}\left(x\right)=g\left(x\right)$
and 
\[
\mathcal{A}\left(y\right)\in\mathcal{A}\left(f^{-1}\left(x\right)\right)\subset\partial_{\gamma}g\left(x\right).
\]
By the definition of $\gamma$-subdifferentials, we have 
\[
\left(\exists C>0, \forall z\in H\right)\qquad\theta_{1}\left(z\right)-\theta_{1}\left(x\right)\geq g\left(z\right)-g\left(x\right)\geq \left\langle \mathcal{A}\left(y\right),z\right\rangle -\left\langle \mathcal{A}\left(y\right),x\right\rangle -C\Vert z-x\Vert^\gamma,
\]
so $\mathcal{A}\left(y\right)\in\partial_{\gamma}\theta_{1}\left(x\right)$.
Let $\varphi: H\to (-\infty,+\infty],\  \varphi=\theta_{1}-I +C\Vert\cdot-f(y)\Vert^\gamma $, then $x$ is a minimizer of $\arg\min_{z\in H}D\left(z,y\right)+\varphi\left(z\right)$.

The second statement can be proved analogously as in Corollary \ref{cor: lsc prox relation general} using $\partial_{\gamma} g(x) \cap \partial_{\gamma}(\varphi+I)(x)\neq \emptyset$

The third and fourth statements can be proved by following the lines of the proof of Theorem \ref{thm: prox relation general} and Corollary \ref{cor: lsc prox relation general} with the necessary changes mentioned in the proof of the first two statements above.
\end{proof}
Observe that the mapping $f$ which appears in Corollary \ref{cor: paraconvex prox relation general} plays the same role as the proximity operator in Theorem \ref{thm: prox relation general}, that is why we keep the same terminology.


Recently in \cite{bednarczuk2025proximal}, by exploiting the structure of the $\Phi_{lsc}^\mathbb{R}$-functions, we introduce $\Phi_{lsc}^\mathbb{R}$-proximal operator which mimics the classical proximal operator in convex analysis. 
In the example below, by using Corollary \ref{cor: lsc prox relation general}, we characterize $\Phi_{lsc}^\mathbb{R}$-proximity operator.

\begin{example}
\label{exm1}
\cite[Formula 22]{bednarczuk2025proximal} The $\Phi_{lsc}^{\mathbb{R}}$-proximal operator of $\varphi:H\to(-\infty,+\infty]$
at $y\in H$ has the form: 
\begin{equation}
\label{eq: lsc prox operator form}
\mathrm{prox_{\varphi}^{lsc} } (y):= \{ x\in\left(\mathcal{J}+\partial_{\Phi}\varphi\right)^{-1}\mathcal{J}\left(y\right)\},
\end{equation}
where $\mathcal{J}:=\partial_{\Phi}\left(\frac{1}{2}\Vert\cdot\Vert^{2}\right)$ and the inverse image of $\mathcal{J}(y)$ under $\left(\mathcal{J}+\partial_{\Phi}\varphi\right)$ is understood as follows.
\begin{equation}
\label{eq: lsc prox case 2}
(\exists x\in H) \quad \left(\mathcal{J}+\partial_{\Phi}\varphi\right)\left(x\right)\cap \mathcal{J}\left(y\right)\neq \emptyset.
\end{equation}
For more on this topic, the readers are referred to \cite[Chapter VI.1]{berge1877topological}.
Then there exists $\phi_{0}\in\mathcal{J}\left(y\right)$
such that 
\[
\phi_{0}\in\partial_{\Phi}\left(\varphi+\frac{1}{2}\Vert\cdot\Vert^{2}\right)\left(x\right).
\]
By definition of $\Phi_{lsc}^\mathbb{R}$-subdifferentials, we obtain 
\[
\left(\forall z\in H\right)\qquad\varphi\left(z\right)+\frac{1}{2}\Vert z\Vert^{2}-\phi_{0}\left(z\right)\geq\varphi\left(x\right)+\frac{1}{2}\Vert x\Vert^{2}-\phi_{0}\left(x\right),
\]
so $x$ is the global minimizer of $\varphi\left(z\right)+D(z,y)$ where 
\begin{align}
D\left(z,y\right)& =I\left(z\right)-\left\langle \mathcal{A}\left(y\right),z\right\rangle _{\Phi}+J\left(\phi_{0}\right), \label{eq: coupling D lsc}\\
I\left(z\right) & =\frac{1}{2}\Vert z\Vert^{2}, \quad
\mathcal{A}(y) = \phi_0,
\nonumber
\end{align}
and $J\equiv 0$.


Conversely, let $\phi_0\in \mathcal{J}(x_0)$ and $x=f(x_0)\in H$, if 
\[
x\in \arg\min_{z\in H} \varphi(z) +D(z,x_0),
\]
then $\phi_0 \in \partial_{\Phi} (\varphi+\frac{1}{2}\Vert\cdot\Vert^2)(x)$ and so 
\[
\left(\mathcal{J}+\partial_{\Phi}\varphi\right)\left(x\right)\cap \mathcal{J}\left(x_0\right)\neq \emptyset.
\]
Therefore, $x=f(x_0)$ is a $\Phi_{lsc}^\mathbb{R}$-proximal point of $\varphi$ at $x_0$.
\end{example}
From the above example, we define $f:H\to H$ to be a selector of $\Phi_{lsc}^\mathbb{R}$-proximity operator of $\varphi$ \eqref{eq: lsc prox operator form} i.e. for $y\in H, x=f(y)$ if and only if $x\in \mathrm{prox}_{\varphi}^{lsc} (y)$.
Then $\Phi_{lsc}^\mathbb{R}$-proximity operator as defined in Example \ref{exm1} can be equivalently expressed with the help of Theorem \ref{thm: prox relation general} as shown in the corollary below.

Before that, let us define the inverse image of $f$ at $x$ is defined as $f^{-1}(x)=\{y\in \mathrm{dom }f: f(y) =x\}$ and denote $\mathcal{J}(x)= \partial_\Phi (\frac{1}{2}\Vert\cdot\Vert^2)(x)$ for any $x\in H$.
\begin{corollary}
\label{cor: lsc prox relation}
Let $H$ be a Hilbert space and $f:H\to H$ be a mapping. Let $J:\Phi_{lsc}^\mathbb{R}\to(-\infty,+\infty],\varphi :H\to(-\infty,+\infty]$ be given functions.
Consider $x_0\in H$.  
The following are equivalent.
\begin{enumerate}[label=\roman*.]
    \item $f$ is a selector of $\Phi_{lsc}^\mathbb{R}$-proximity operator of $\varphi$ at $x_0$. 
    \item There exists a $\Phi_{lsc}^\mathbb{R}$-convex $g:H\to (-\infty,+\infty]$ such that the intersection $\partial_{\Phi} g(x) \cap \mathcal{J}(f^{-1}(x))$ is non-empty for any $x\in \mathrm{Im} f$.
    \item $f(x_0)$ is a global minimizer of  $\varphi(z) +D(z,x_0)$ where $D(z,x_0)$ is defined in \eqref{eq: coupling D lsc}.
    \item Let $E\subset\mathrm{Im}f$ be polygonally connected. There exists a constant $K\in\mathbb{R}$ such that 
\[
g\left(z\right)=\frac{1}{2}\Vert z\Vert^2+\varphi\left(z\right)+K,
\]
for each $z\in E$.
\end{enumerate}
\end{corollary}

\begin{remark}
The equivalence between $\Phi_{lsc}^\mathbb{R}$-subgradient and the proximal subgradient ($\gamma=2$) (see \cite[Proposition 3.12]{syga2019global}) implies that the results in Corollary \ref{cor: lsc prox relation} still hold true for proximal subdifferentials.
\end{remark}

In the example below, we apply Corollary \ref{cor: lsc prox relation}.
\begin{example}
\label{ex2:-}
\begin{itemize}
    \item Let $H$ be a Hilbert space and $\Phi=\Phi_{lsc}^\mathbb{R}$, consider the function $\varphi\left(x\right)=\Vert x\Vert^{2}$ and $x_{0}\in H$. 
    For $\phi_{0}\in\mathcal{J}(x_0) =\partial_{\Phi}\left(\frac{1}{2}\Vert x_0\Vert^{2}\right)$,
the mapping $f:H\to H$ define as $x=f(x_0)$, where $x$ satisfies \eqref{eq: lsc prox operator form}. Then by \cite[Theorem 4]{bednarczuk2025proximal}, $x$ is a solution of the minimization problem 
\[
\min_{x\in H} \varphi(x) +(\frac{1}{2}+a_0)\Vert x-x_0\Vert^2 = \min_{x\in H}\Vert x\Vert^2 +(\frac{1}{2}+a_0)\Vert x-x_0\Vert^2.
\]
Solving the above problem, we obtain
\[
f\left(x_0\right)=\frac{1+2a_{0}}{3+2a_{0}}x_{0},\quad \text{ for some } \phi_{0}=\left(a_{0},\left(1+2a_{0}\right)x_{0}\right), a_0\geq -\frac{1}{2}.
\]
Notice that $3+2a_{0}>0$ for any $\phi_{0}\in\mathcal{J}\left(x_{0}\right)$.
Moreover, using the formula of $x=f\left(x_0\right)$, $x$ is the global minimizer of the problem 
\begin{equation}
\label{eq: lsc prox standard form}
\varphi\left(z\right)+D(z,x_0)
=\frac{3+2a_0}{2}\Vert z - \frac{1+2a_0}{3+2a_0}x_0\Vert^2,
\end{equation}
where $D(z,x_0)$ is defined as in \eqref{eq: coupling D lsc}.
with $I(z) = \frac{1}{2}\Vert z\Vert^2, \mathcal{A}(x_0)=\phi_0, J(x_0) = \left(\frac{1+2a_0}{3+2a_0}\right)^2 \Vert x_0\Vert^2.$
    \item For $H=\mathbb{R}, \Phi=\Phi_{lsc}^\mathbb{R}$, $h\left(x\right)=\left|x\right|$, we calculate 
\[
\partial_{\Phi}h\left(x\right)=\begin{cases}
a\geq0,u=2ax+1 & x>0\\
a\geq0,u=2ax-1 & x<0\\
a\geq0,-1\leq u\leq1 & x=0
\end{cases}.
\]
With $\mathcal{J}\left(x\right):=\partial_{\Phi}h \left(x\right)$
and $\varphi\left(x\right)= x^{2}$, the proximity operator
$f:H\to \mathbb{R}$ of $\varphi$ at $x_{0}$ has the
form 
\[
f\left(x_0\right)=\begin{cases}
\frac{2a_{0}x_{0}}{1+2a_{0}} & x_{0}\neq0\\
0 & x_0=0.
\end{cases}, \quad \text{ for some } \phi_0\in \mathcal{J}(x_0).
\]

\item We consider $a>-1/2$, and the problem of minimizing the function $\varphi+D\left(\cdot,\phi_{0}\right)$ over $H$ where 
\begin{align*}
D\left(x,\phi_0\right) & =I\left(x\right)-\phi_{0}\left(x\right)+J\left(\phi_{0}\right),
\end{align*}
with $I\left(x\right)  =\left(\frac{1}{2}+a\right)\Vert Mx\Vert^{2}-a\Vert x\Vert^{2},
\phi_{0}  =\left(a,\left(1+2a\right)M^{*}y\right),
J\left(\phi_{0}\right)  =\left(\frac{1}{2}+a\right)\Vert y\Vert^{2}$ and $M:H\to H$ is a linear operator.
Hence, this problem is equivalent to the regularized problem $\min_{x\in H}\varphi(x) +(\frac{1}{2}+a)\Vert Mx -y\Vert^2$.
A solution $x$ to this problem is determined by a proximal mapping $f$, $x=f(\phi_0)$ i.e. according to Theorem \ref{thm: prox relation general}-1(i), $f(\phi_0) \in \argmin\ \varphi(\cdot) + D(\cdot,\phi_0)$.
This is a common model in inverse problems which has a wide range of applications, especially in image processing and machine learning.

\end{itemize}

\end{example}

\begin{example}
We demonstrate the flexibility of Corollary \ref{cor: paraconvex prox relation general} when applied to $\gamma$-paraconvex function. As in Example 3.15 \cite[Example 3.15]{rahimi2024projected}, consider the function $\varphi:\mathbb{R}\to\mathbb{R}$,
\[
\varphi (x) = \begin{cases}
1-|x|^\gamma & -1 \leq x\leq 1\\
x^2-1 & \text{otherwise}
\end{cases},
\]
where $1<\gamma <2$. 
According to Definition \ref{def:paraconvex}, this is a $\gamma$-paraconvex function. Taking $x_0\in \mathbb{R}$, we try to find the global minimizer of $\varphi(x)+(x-x_0)^2$ or $\varphi(x)+|x-x_0|^\gamma$ (as in classical proximal operator), which is impossible to compute explicitly. With Corollary \ref{cor: paraconvex prox relation general}, by choosing 
$I(x)=\begin{cases}
|x|^\gamma & -1\leq x\leq 1\\
x^2 & \text{Otherwise},
\end{cases}$ and $J\equiv 0$, 
the proximal operator of $\varphi$ at $x_0$, $f:\mathbb{R}\to\mathbb{R}$, is the global minimizer of $\varphi(x)+D(x,x_0)$ which has an explicit form as
\[
f(x_0) = \begin{cases}
1 & 0\leq x_0 \leq 4\\
-1 & -4 \leq x_0 \leq 0\\
\frac{x_0}{4} & \text{Otherwise}
\end{cases}.
\]
Then $x_0 \in \partial_\gamma g(x)$, where $g(z) = \varphi(z) +I(z)$ by Corollary \ref{cor: paraconvex prox relation general}, $x\in \mathrm{Im }f = (-\infty,-1]\cup [1,+\infty]$. 
\end{example}

\begin{example}
In this example, we apply Corollary \ref{cor: lsc prox relation} to construct function $\psi$ when $f$ and $\varphi$ are known. Let us consider the hard-thresholding function on the real line
\[
H_{\lambda}(x)=\begin{cases}
x & \text{if }|x|>\lambda\\
0 & \text{if }|x|\leq\lambda
\end{cases},
\]
where $\lambda>0$ is a parameter. Then for $\gamma>0,x_0\in \mathbb{R}$, the classical proximal operator \eqref{eq: classic prox} of $H_\lambda$ is a global minimizer of $H_\lambda(x)+\frac{1}{2\gamma}\Vert x-x_0\Vert^2$ which we denote $f(x_0)$. Depending on the relation between $\gamma$ and $\lambda$, we obtain the following cases:
\begin{itemize}
    \item $\gamma\leq2\lambda$: 
    \[
    f(x_0)=\begin{cases}
x_{0}-\gamma & x_{0}<-\lambda+\gamma \text{ or } x_{0}>\lambda+\gamma+\sqrt{2\lambda\gamma}\\
x_{0} & -\lambda+\gamma\leq x_{0}\leq\lambda\\
\lambda & \lambda<x_{0}\leq\lambda+\gamma+\sqrt{2\lambda\gamma}
\end{cases}.
    \]
\item $\gamma>2\lambda$: 
\[
f(x_0)=\begin{cases}
x_{0}-\gamma & x_{0}<\lambda+\gamma-\sqrt{2\lambda\gamma} \text{ or } x_{0}>\lambda+\gamma+\sqrt{2\lambda\gamma}\\
\lambda & \lambda+\gamma-\sqrt{2\lambda\gamma}<x_{0}\leq\lambda+\gamma+\sqrt{2\lambda\gamma}\\
\end{cases}.
\]
\end{itemize}
The function $f(x_0)$ is discontinuous. For simplicity, let us take $\gamma=3,\lambda=1$, we reconstruct the function $\psi: E\to \mathbb{R}$, where $E\subset \mathrm{Im }f$ is polygonally connected, 
\begin{align}
    \psi (y) & = y f(y) -\frac{1}{2}\Vert f(y)\Vert^2 -H_\lambda (f(y)) \nonumber\\
    & = \begin{cases}
\frac{y^{2}}{6}+\frac{3}{2}-y & y<4-\sqrt{6} \text{ or } y>4+\sqrt{6} \nonumber\\
\frac{y}{3}-\frac{1}{6} & \text{ otherwise}
\end{cases},
\end{align}
where $\mathbb{B}=\mathrm{Id},I(x) = J(x)=\frac{1}{2}x^2$ and $\varphi = H_\lambda$ as in Corollary \ref{cor: paraconvex prox relation general}-4 with $G=0, \gamma=2$.
Observe that $\psi$ is a convex continuous function on $\mathbb{R}$.

\end{example}
\subsection{Some variants to the function $D(\cdot )$}
We end this section with some remarks.
We can employ $\Phi$-proximity operator as a tool to find minimizers of function $\varphi$ mentioned in Theorem \ref{thm: prox relation general}. 
To do this, we need to assume
\begin{itemize}
    \item $D(x,\phi)\geq 0$ for all $x\in H,\phi\in \Phi$. This guarantees that $\min \varphi \leq \min \varphi +D(\cdot,\phi)$.
    \item For every $x_0\in H$, there exists $\phi$ such that $D(x_0,\phi) = 0$. Then $\min_{x\in H} \varphi(x) +D(x,\phi) \leq \varphi(x_0)$. This helps us to express the envelope $\min \varphi +D(\cdot,\phi)$ as a function of $x_0$.
\end{itemize}
Combining these assumptions, we see that $\varphi+D(\cdot,\phi)$ and $\varphi$ have the same optimal values and minimizers for some $\phi\in\Phi$. 
An example of such $D(x,\phi)$ is to use $\phi \in \partial_\Phi I(x_0)$ for $x_0\in H$ and set 
\begin{equation}
\label{eq: phi-breman our}
D(x,\phi,x_0) := (I(x) -\phi(x)) - (I(x_0)-\phi(x_0)).
\end{equation}
This definition coincides with the one in \cite{millan2024global}, where the authors construct $D(x,\phi)$ in the same way and call it abstract Bregman divergence
\begin{equation}
\label{eq: phi-bregmann}
    D_\phi^\lambda \left(x,y\right)=\phi(x)-\phi(y) - [\lambda(x)-\lambda(y)],
\end{equation}
where $\lambda\in \partial_\Phi \phi(y)$. In \cite{millan2024global}, the $\Phi$-proximity operator is defined based on $D_\phi^\lambda$.
In fact, when $\Phi=\Phi_{lsc}^\mathbb{R}$, the definition of $\Phi$-proximity operator in \cite{millan2024global} aligns with our definition in \cite{bednarczuk2025proximal}. For example, with $I(x) = \frac{1}{2}\Vert x\Vert^2,\phi\in \partial_{\Phi}I(x_0)$, we get 
\begin{align*}
D(x,\phi,x_0) &= (I(x) -\phi(x)) - (I(x_0)-\phi(x_0)) \\
& =(\frac{1}{2}\Vert x\Vert^2 -\phi(x)) - (\frac{1}{2}\Vert x\Vert^2-\phi(x_0)) = (\frac{1}{2}+a_0)\Vert x-x_0\Vert^2,
\end{align*}
where the formula of $\phi$ is given in Example \ref{exm1}. Hence, $\arg\min \varphi +D(x,\phi,x_0)$ is exactly the classical proximal operator as defined by \eqref{eq: classic prox} of $\varphi$ at $x_0$ with parameter $a_0\geq-1/2$ when $\varphi$ is convex.

Moreover, Theorem \ref{thm: prox relation general} allows us to extend the definition of the proximity operator using a more general nonconvex Bregman distance. 
For example, one can use Descent Lemma \cite[Lemma 2.64]{Bau2011} (for any class of differential functions with Lipschitz continuous gradient) and define the distance
\begin{align*}
\hat{D}(x,\phi,x_0) & = I(x)-I(x_0) - \langle\nabla I(x),x-x_0\rangle +\frac{L_I}{2}\Vert x-x_0\Vert^2\\
& = I(x)-I(x_0) -[\phi(x)-\phi(x_0)]\geq 0,
\end{align*}
where $I:H\to (-\infty,+\infty]$ is differentiable with Lipschitz continuous gradient and $\phi(z) = \frac{L_I}{2}\Vert z\Vert^2 +\langle\nabla I(x)-L_I x,z\rangle$ with $L_I>0$ is the Lipschitz constant of $\nabla I$.

\section{Continuity of proximity operator}
\label{sec: smooth prox}

Let us consider the assumption in Theorem \ref{thm: prox relation general} with $\Phi=H$ and $\mathcal{A}=\mathcal{B}=\mathrm{Id}$, then for $y\in H, D(x,y)=\tilde{D}(x,y)$ and we have the equivalence between all statements in Theorem \ref{thm: prox relation general}.
\begin{corollary}
\label{cor: smooth prox paraconvex}
Let $H$ be Hilbert and $\Phi=H,\mathcal{A}=\mathcal{B}=\mathrm{Id}$, consider $f:H\to H$ and $I,J:H\to (-\infty,+\infty]$ to be Fr\'echet differentiable with the Fr\'echet derivative $\nabla\cdot$. 
Under these setting, let $\varphi,g,\psi$ satisfy the equivalences in Theorem \ref{thm: prox relation general} with $\gamma$-subdifferentials for $2\geq \gamma>1$, and let $k\in\mathbb{N}$. 
\begin{enumerate}
    \item For an open set $V\subset H$, we have $f\in C^k(V) \Leftrightarrow \psi \in C^{k+1}(V)$ and $f(y)=\nabla \psi(y)$ for all $y\in V$, where $C^k(V)$ is the class of all $k$-continuously differentiable functions defined on $V$.
    \item Consider a polygonally connected set $V\subset \mathrm{Im}f$, the following are equivalent.
    \begin{enumerate}[label=\roman*.]
        \item $\varphi\in C^{k+1}(V)$.
        \item $g\in C^{k+1}(V)$.
        \item The restriction $\tilde{f}$ of $f$ to $f^{-1}(V)$ is injective and $\tilde{f}^{-1}\in C^k (V)$.
        If one of the statement 2-(i) or 2-(ii) hold, then $\tilde{f}$ is a bijection between $f^{-1}(V)$ and $V$. Furthermore, $\tilde{f}^{-1}(x) = \nabla g(x) = \nabla \varphi(x)+\nabla I(x)$ for all $x\in V$.
    \end{enumerate}
\end{enumerate}
\end{corollary}
\begin{proof}
(1) The equivalence comes from Lemma \ref{lem: paraconvex smooth selection} and Corollary \ref{cor: paraconvex prox relation general}  with $f$ as a selection of the $\gamma$-subdifferentials $\partial_{\gamma}\psi$ for $\psi$ given in Corollary \ref{cor: paraconvex prox relation general}-(3,4). 
Hence, it holds for $k=0$ and any $y\in V$. 
For $k\in\mathbb{N}, k\geq 1$, we set $\psi_k(y) = \nabla^k \psi(y)$ and $f_k(y) = \nabla^k f(y)$ where $\nabla^k \psi$ is the $k$-derivative of $\psi$. We assume that $\nabla^k f(y) = \nabla^{k+1} \psi(y)$. Then $\nabla^{k+1}f(y) = \nabla (\nabla^k f)(y) = \nabla (\nabla^{k+1} \psi)(y) = \nabla^{k+2} \psi(y)$. 
Hence, $\nabla^k f(y) = \nabla^{k+1}\psi(y)$ for any $k\in\mathbb{N}$ by induction.

(2) The proof follows the lines of the proof in \cite[Corollary 6]{gribonval2020characterization} by using Lemma \ref{lem: paraconvex smooth selection}, Lemma \ref{lem: cont selection smooth} and Corollary \ref{cor: paraconvex prox relation general}.
\end{proof}

The claims of Corollary \ref{cor: smooth prox paraconvex} relies on continuous selector of proximal subgradient and the proximal subdifferentials is unique for differentiable functions (see Remark \ref{rmk: weak phi subdiff} and \ref{rmk: gamma subgrad gateaux}). 
For $\Phi$-convex functions, $\Phi$-subdifferentials is not unique even in the differentiable case. However, when the class $\Phi$ consists of differentiable functions, we obtain the following.

\begin{proposition}
\label{prop: diff and subgrad}
Let $\varphi:H\to (-\infty,+\infty]$ be proper G\^ateaux differentiable at $x\in H$ and the class $\Phi$. If $\partial_{\Phi}\varphi (x)\neq \emptyset$, then we have $\nabla \varphi(x) = \nabla \phi (x)$ for any $\phi\in \partial_{\Phi}\varphi (x)$.
\end{proposition}
\begin{proof}
Let $\varphi$ be G\^ateaux differentiable at $x$ and let $\phi\in \partial_{\Phi} \varphi (x)$. We have 
\[
(\forall y\in H) \quad \varphi(y)-\varphi (x) \geq \phi (y) -\phi (x).
\]
Let $y=x+th$ where $t>0, h\in H$, we have 
\begin{equation}
    \varphi(x+th) -\varphi(x) \geq \phi (x+th) -\phi (x).
\end{equation}
Divide both sides by $t$ and taking the limit $t\to 0_+$, we obtain 
\[
\langle \nabla \varphi(x) ,h\rangle \geq \langle \nabla \phi (x) ,h\rangle,
\]
for any $h\in H$. Replacing $h$ by $-h$, we obtain
\[
-\langle \nabla \varphi(x) ,h\rangle \geq \langle -\nabla \phi (x) ,h\rangle,
\]
which implies that $\nabla \varphi(x) = \nabla \phi (x)$. Since there is no restriction on $\phi$, this holds for any $\phi \in \partial_{\Phi} \varphi (x)$.
\end{proof}

We give an example of how to use Corollary \ref{cor: paraconvex prox relation general} and Corollary \ref{cor: smooth prox paraconvex} to reconstruct the function $\varphi$ if we know its proximity operator $f$. 
\begin{example}
Consider $y\in\mathbb{R}$ and $f\left(y\right)=y^{2/3}$, then for $\psi\left(y\right)$ such that $\nabla\psi\left(y\right)=f\left(y\right)$ (by Corollary \ref{cor: smooth prox paraconvex}-1), we have $\psi\left(y\right)=\frac{3}{5}y^{5/3}$. 
Let $I\left(x\right)=\frac{2}{3}x^{5/3}$, from Corollary \ref{cor: paraconvex prox relation general}-4, we reconstruct $\varphi\left(x\right)$ in the form 
\begin{align*}
\varphi\left(f\left(y\right)\right) & =yf\left(y\right)-I\left(f\left(y\right)\right)-\psi\left(y\right)\\
 & =y^{5/3}-\frac{2}{3}y^{10/9}-\frac{3}{5}y^{5/3} =\frac{2}{5}y^{5/3}-\frac{2}{3}y^{10/9}.
\end{align*}
Then, we recover $\varphi\left(x\right)$ with 
\[
\varphi\left(x\right)=\frac{2}{5}|x|^{5/2}-\frac{2}{3}x^{5/3}.
\]
For $y\in\mathbb{R}$, the problem $\varphi(x)+I(x)-x y = \frac{2}{5}|x|^{5/2}-xy$ is convex and has a global solution $x=y^{2/3}$.
\end{example}

\section{Auxiliary results}
\label{sec: supporting proof}
In this section, we present several properties of $\Phi_{lsc}^\mathbb{R}$-convex functions and $\gamma$-paraconvex functions. These results are essential for the proof of the main result in Section \ref{sec: prox main result}.

We provide an extension to the results of \cite[Corollary 9]{gribonval2020characterization}.

\begin{lemma}
\label{lem: diff const wc}
\begin{itemize}
    \item Consider $a_{1},a_{2}:\mathbb{R}\to(-\infty,+\infty]$ to be $\rho_{1}\left(\rho_{2}\right)$-weakly convex such that $\mathrm{dom}a_{i}=\mathrm{dom}\ \partial_{(2,\rho_{i})} a_{i}=\left[0,1\right]$ and $\partial_{(2,\rho_{1})}a_{1}\left(t\right)\cap\partial_{(2,\rho_{2})}a_{2}\left(t\right)\neq\emptyset$ on $\left[0,1\right]$. Then there exists a constant $K\in\mathbb{R}$ such that $a_{1}\left(t\right)-a_{2}\left(t\right)=K$ on $\left[0,1\right]$.
    \item Let $\theta_{1},\theta_{2}:H\to(-\infty,+\infty]$ be proper and $E\subset H$ a non-empty polygonally connected set. 
Assume that $\rho_{1},\rho_{2}>0$, for each $z\in E,\partial_{(2,\rho_{1})}\theta_{1}\left(z\right)\cap\partial_{(2,\rho_{2})}\theta_{2}\left(z\right)\neq\emptyset$; then there exists a constant $K\in\mathbb{R}$ such that $\theta_{1}\left(x\right)-\theta_{2}\left(x\right)=K$ for all $x\in E$. 
\end{itemize}

\end{lemma}
\begin{proof}
We only prove the first statement. The second statement can be deduced from \cite[Corollary 9]{gribonval2020characterization}.
Without loss of generality, we assume that $\rho_{1}\leq\rho_{2}$, then $\bar{a}_{1}:=a_{1}+\frac{\rho_{2}}{2}\Vert\cdot\Vert^{2}$ is convex on $\left[0,1\right]$. Set $\bar{a_{2}}=a_{2}+\frac{\rho_{2}}{2}\Vert\cdot\Vert^{2}$,
we need to prove that $\partial\bar{a}_{1}\left(t\right)\cap\partial\bar{a}_{2}\left(t\right)\neq\emptyset$ on $\left[0,1\right]$. 
By assumption, let $u\in\partial_{(2,\rho_{1})}a_{1}\left(t\right)\cap\partial_{(2,\rho_{2})}a_{2}\left(t\right)$, for any $y\in\mathbb{R}$, we have
\begin{align*}
a_{1}\left(y\right)-a_{1}\left(t\right) & \geq\left\langle u,y-t\right\rangle -\frac{\rho_{1}}{2}\Vert y-t\Vert^{2}\\
 & \geq\left\langle u,y-t\right\rangle -\frac{\rho_{2}}{2}\Vert y-t\Vert^{2} \geq\left\langle u+\rho_{2}t,y-t\right\rangle -\frac{\rho_{2}}{2}\left(\left\Vert y\right\Vert ^{2}-\left\Vert t\right\Vert ^{2}\right).
\end{align*}
Hence $u+\rho_{2}t\in\partial\bar{a}_{1}\left(t\right)$, we also
obtain that for $u\in\partial_{(2,\rho_{2})}a_{2}\left(t\right),u+\rho_{2}t\in\partial\bar{a}_{2}\left(t\right)$.
Thus, $\partial\bar{a}_{1}\left(t\right)\cap\partial\bar{a}_{2}\left(t\right)\neq\emptyset$.
By \cite[Lemma 4]{gribonval2020characterization}, there exists a constant $K\in\mathbb{R}$
such that 
\begin{align*}
\bar{a}_{1}\left(t\right)-\bar{a}_{2}\left(t\right) & =K \quad \text{or } a_{1}\left(t\right)-a_{2}\left(t\right)  =K.
\end{align*}
\end{proof}
We extend the result of Lemma \ref{lem: diff const wc} for $\Phi_{lsc}^\mathbb{R}$-convex functions.
\begin{lemma}
\label{lem: lsc conv poly-connected}
\begin{itemize}
    \item Consider $\theta_{1},\theta_{2}:\mathbb{R}\to(-\infty,+\infty]$ to be proper $\Phi_{lsc}^\mathbb{R}$-convex such that $\mathrm{dom}\  \theta_{i}=\mathrm{dom}\ \partial_{lsc}^\mathbb{R}\theta_{i}=\left[0,1\right]$ and $\partial_{lsc}^\mathbb{R}\theta_{1}\left(t\right)\cap\partial_{lsc}^\mathbb{R}\theta_{2}\left(t\right)\neq\emptyset$ on $\left[0,1\right]$. Then there exists a constant $K\in\mathbb{R}$ such that $\theta_{1}\left(t\right)-\theta_{2}\left(t\right)=K$ on $\left[0,1\right]$.
    \item Let $\theta_{1},\theta_{2}:H\to(-\infty,+\infty]$ be proper and $E\subset H$ a non-empty polygonally connected set. 
Assume that for each $z\in E,\partial_{lsc}^\mathbb{R}\theta_{1}\left(z\right)\cap\partial_{lsc}^\mathbb{R}\theta_{2}\left(z\right)\neq\emptyset$; then there exists a constant $K\in\mathbb{R}$ such that $\theta_{1}\left(x\right)-\theta_{2}\left(x\right)=K$ for all $x\in E$. 
\end{itemize}

\end{lemma}
\begin{proof}
As $\mathrm{dom}\ \partial_\Phi \theta_{i}=\left[0,1\right]$, for every $t\in\left[0,1\right]$ there exists $\phi_{t}=\left(a_{t},u_{t}\right)\in\partial_{\Phi}\theta_{1}\left(t\right)\cap\partial_{\Phi}\theta_{2}\left(t\right)$
so that for all $z\in\mathbb{R}$,
\begin{align*}
\left(\theta_{1}\left(z\right)+a_{t}\Vert z\Vert^{2}\right)-\left(\theta_{1}\left(t\right)+a_{t}\Vert t\Vert^{2}\right) & \geq\left\langle u_{t},z-t\right\rangle ,\\
\left(\theta_{2}\left(z\right)+a_{t}\Vert z\Vert^{2}\right)-\left(\theta_{2}\left(t\right)+a_{t}\Vert t\Vert^{2}\right) & \geq\left\langle u_{t},z-t\right\rangle .
\end{align*}
Set $\tilde{a}:=\sup_{t\in[0,1]}a_{t}$ which is finite as we are taking $t\in[0,1]$, and $\bar{\theta}_{i}:=\theta_{i}+\tilde{a}\Vert\cdot\Vert^{2}$. We have $\partial \bar{\theta}_1 (t) \cap \partial \bar{\theta}_2(t) \neq \emptyset$ on $[0,1]$ which infers $\bar{\theta}_{i}$ is convex on $\left[0,1\right]$. 
Using Lemma \ref{lem: diff const wc}, there exists a constant $K\in\mathbb{R}$ such that $\bar{\theta}_{1}\left(t\right)-\bar{\theta}_{2}\left(t\right)=K$.
Switching back to $\theta_{1},\theta_{2}$ we finish the proof.
\end{proof}

The results of Lemma \ref{lem: diff const wc} can be generalized to $\gamma$-paraconvex functions. 

\begin{lemma}
\label{lem: paraconv subdiff real line}
Let $C>0$ and $1<\gamma \leq 2$.
\begin{itemize}
    \item Consider $a_{1},a_{2}:\mathbb{R}\to(-\infty,+\infty]$ to be proper $\gamma$-paraconvex such that $\mathrm{dom}\ a_{i}=\mathrm{dom}\ \partial_{\gamma}a_{i}=\left[0,1\right]$ and $\partial_{\gamma}a_{1}\left(t\right)\cap\partial_{\gamma}a_{2}\left(t\right)\neq\emptyset$ on $\left[0,1\right]$. Then there exists a constant $K\in\mathbb{R}$ such that $a_{1}\left(t\right)-a_{2}\left(t\right)=K$ on $\left[0,1\right]$.
    \item Let $\theta_{1},\theta_{2}:H\to(-\infty,+\infty]$ be proper and $E\subset H$ a non-empty polygonally connected set. 
Assume that for each $z\in E,\partial_{\gamma}\theta_{1}\left(z\right)\cap\partial_{\gamma}\theta_{2}\left(z\right)\neq\emptyset$; then there exists a constant $K\in\mathbb{R}$ such that $\theta_{1}\left(x\right)-\theta_{2}\left(x\right)=K$ for all $x\in E$. 
\end{itemize}
\end{lemma}
\begin{proof}
We only prove the first statement. The second one can be proved analogously \cite[Corollary 9]{gribonval2020characterization}. As $\partial_{\gamma} a_i (x) \neq \emptyset$ for all $x\in [0,1]$, it is $\gamma$-paraconvex (see \cite[Proposition 3.13]{rahimi2024projected}. By \cite[Theorem 3.8]{rahimi2024projected}, it is locally Lipschitz continuous on $(0,1)$ and so is Fr\'echet differentiable almost everywhere on $(0,1)$ by Rademacher's Theorem (see \cite{evans2025measure}).
Hence, there exist a countable set $B_i\subset [0,1]$ with measure zero such that   $\gamma$-subdifferentials of $a_i$ coincide with Clarke subdifferentials which is the derivative of $a_i$ except on $(0,1)\backslash B_i$ (see \cite[Theorem 3.1]{jourani1996subdifferentiability}). The continuity of $a_i$ can be extended to $[0,1]$ by Proposition \ref{prop: paraconvex-cont end point}. Hence, $a_i$ is differentiable on $[0,1]\backslash B_i$.
Combining with the assumption, $a_1$ and $a_2$ has the same derivative on $(0,1)\backslash (B_1\cup B_2)$ which implies that there exists a constant $K\in\mathbb{R}$ such that $a_1(x)-a_2(x)=K$ for all $x\in [0,1]$.

The proof of the second statement is analogous to Lemma \ref{lem: diff const wc}.
\end{proof}

\subsection{Continuity of proximity operator auxiliary results}
We present the supporting results to Section \ref{sec: smooth prox} and some properties of paraconvex functions. 
\begin{lemma}
\label{lem: paraconvex smooth selection}
Let $f:H\to(-\infty,+\infty]$ be proper $\gamma$-paraconvex where
$1<\gamma\leq2$ and $G:H\to H$. Consider an open set $X\subset\mathrm{dom}\ G\cap\mathrm{dom}\ f$, assuming $\partial_{\gamma}f\left(x\right)\neq\emptyset$ for all $x\in X,C>0$ and $G\left(x\right)\in\partial_{\gamma}f\left(x\right).$
Then the followings are equivalent:
\begin{enumerate}
\item $G$ is continuous on $X$.
\item $f$ is continuously differentiable on $X$.
\end{enumerate}
\end{lemma}
\begin{proof}
From (2.) to (1.) is clear, so we just need to prove (1.) implies
(2.). Let $x\in X$, for any $y$ in the neighborhood of $x$, there exists $C>0$ such that
\begin{align*}
-C\left\Vert y\right\Vert ^{\gamma} & \leq f\left(x+y\right)-f\left(x\right)-\left\langle y,G\left(x\right)\right\rangle \\
 & \leq\left\langle y,G\left(x+y\right)\right\rangle -\left\langle y,G\left(x\right)\right\rangle +C\left\Vert y\right\Vert ^{\gamma}\\
 & \leq\left\Vert y\right\Vert \left\Vert G\left(x+y\right)-G\left(x\right)\right\Vert +C\left\Vert y\right\Vert ^{\gamma}.
\end{align*}
Hence, we have 
\[
-\lim_{\left\Vert y\right\Vert \neq0,\left\Vert y\right\Vert \to0}C\left\Vert y\right\Vert ^{\gamma-1}\leq\lim_{\left\Vert y\right\Vert \neq0,\left\Vert y\right\Vert \to0}\frac{f\left(x+y\right)-f\left(x\right)-\left\langle y,G\left(x\right)\right\rangle }{\left\Vert y\right\Vert }\leq\lim_{\left\Vert y\right\Vert \neq0,\left\Vert y\right\Vert \to0}\left\Vert G\left(x+y\right)-G\left(x\right)\right\Vert +C\left\Vert y\right\Vert ^{\gamma-1}.
\]
By the continuity of $G$, the above inequalities imply that $f$
is differentiable at $x\in X$ and $G\left(x\right)=\nabla f\left(x\right)$
which is continuous.
\end{proof}
The proof above is an imitation of \cite[Proposition 17.41]{Bau2011}. In fact, we can extend \cite[Proposition 17.41]{Bau2011} to $\gamma$-paraconvex function with $\gamma>1$.

\begin{lemma}
\label{lem: cont selection smooth}
Let $f:H\to(-\infty,+\infty]$ be proper $\gamma$-paraconvex where
$1<\gamma\leq2$ and $x\in \mathrm{int}\ \mathrm{dom}\ f$.
The following are equivalent:
\begin{enumerate}[label=\roman*.]
\item $f$ is continuously differentiable on $X$.
\item Every selection of $\partial_{\gamma} f$ is continuous on $X$.
\item There exists a selection of $\partial_{\gamma}f$ that is continuous at $x$.
\end{enumerate}
\end{lemma}
\begin{proof}
(i)$\Rightarrow$(ii): Assume by contradiction, set $u=\nabla f\left(x\right)$, there exist sequences $\left(x_{n},u_{n}\right)_{n\in\mathbb{N}}\subset\mathrm{gra}\ \partial_{\gamma}f$ and $\varepsilon>0$ such that $x_{n}\to x$ and $\Vert u_{n}-u\Vert>\varepsilon$.
On the other hand, the Fr\'echet differentiability at $x$ implies that there exists $\delta>0$ such that for any $y\in B\left(x,\delta\right)$
\[
f\left(x+y\right)-f\left(x\right)-\left\langle y,u\right\rangle \leq\varepsilon\Vert y\Vert.
\]
We can construct a sequence $\left(z_{n}\right)_{n\in\mathbb{N}}$ in the unit ball such that $\left\langle z_{n},u_{n}-u\right\rangle >2\varepsilon+2^{\gamma-1}C\delta^{\gamma}$.
We have 
\begin{align*}
2\varepsilon\delta+2^{\gamma-1}C\delta^{\gamma} & <\delta\left\langle z_{n},u_{n}-u\right\rangle \leq\delta\left\langle z_{n},u_{n}\right\rangle -\delta\left\langle z_{n},u\right\rangle \\
 & \leq f\left(x+\delta z_{n}\right)-f\left(x_{n}\right)-\left\langle x-x_{n},u_{n}\right\rangle +C\left\Vert x-x_{n}+\delta z_{n}\right\Vert ^{\gamma}-\delta\left\langle z_{n},u\right\rangle \\
 & \leq\left[f\left(x+\delta z_{n}\right)-f\left(x\right)-\delta\left\langle z_{n},u\right\rangle \right]+C\left\Vert x-x_{n}+\delta z_{n}\right\Vert ^{\gamma}+f\left(x\right)-f\left(x_{n}\right)-\left\langle x-x_{n},u_{n}\right\rangle \\
 & \leq\varepsilon\Vert\delta z_{n}\Vert+2^{\gamma-1}C\left\Vert x-x_{n}\right\Vert ^{\gamma}+2^{\gamma-1}C\left\Vert \delta z_{n}\right\Vert ^{\gamma}+\left\Vert x-x_{n}\right\Vert \left\Vert u_{n}\right\Vert +f\left(x\right)-f\left(x_{n}\right)\\
 & \rightarrow\varepsilon\delta+2^{\gamma-1}C\delta^{\gamma},
\end{align*}
where the last limit comes from the boundedness of $(u_n)_{n\in\mathbb{N}}$ (see Proposition \ref{prop: paraconv subgrad bounded}) and the continuity of $f$.
This is a contradiction.

We just need to prove (iii)$\Rightarrow$(i): Let $G:H\to H$ be a continuous selection of $\partial_{\gamma} f$ at $x\in \mathrm{int }\mathrm{dom }f$. We can find $\delta>0$ such that $B(x,\delta)\subset \mathrm{dom}\ G \subset\mathrm{dom}\ \partial_{\gamma}f$. Let $y\in B(0,\delta)$ such that $G(x+y)\in \partial_{\gamma} f(x+y)$. We have
\begin{align}
    0 & \leq f(x+y) -f(x) -\langle G(x),y\rangle +C\Vert y\Vert^\gamma \nonumber\\
    & \leq \langle y, G(x+y)-G(x)\rangle +2C\Vert y\Vert^\gamma \nonumber\\
    & \leq \Vert y\Vert \Vert G(x+y)-G(x)\Vert +2C\Vert y\Vert^\gamma.
\end{align}
Since $\gamma>1$, we can divide the above inequality by $\Vert y\Vert$ to obtain
\[
\lim_{\Vert y\Vert \to 0} \frac{f(x+y) -f(x) -\langle G(x),y\rangle}{\Vert y\Vert } =0.
\]
Hence, $f$ is Fr\'echet differentiable at $x$.
\end{proof}
This result helps us to extend \cite[Proposition 17.41]{Bau2011} to $\gamma$-subdifferentials and even to $\Phi_{lsc}^\mathbb{R}$-subdifferentials. The same results for G\^ateaux differential can be obtained as in \cite[Proposition 17.39]{Bau2011}.

\begin{corollary}
The statements in Lemma \ref{lem: cont selection smooth} hold true when replacing $\gamma$-subdifferentials with proximal subdifferentials or $\Phi_{lsc}^\mathbb{R}$-subdifferentials.
\end{corollary}

An auxiliary result related to Lemma \ref{lem: paraconv subdiff real line} is the domain of $\gamma$-paraconvex and $\Phi_\gamma$-convex functions.
We can prove that on the real line, the domain of $\gamma$-paraconvex and $\Phi_\gamma$-convex function $\varphi:\mathbb{R}\to (-\infty,+\infty]$ is an interval.

\begin{lemma}
Let $\varphi:\mathbb{R}\to(-\infty,+\infty]$ be proper function. We have the following:
\begin{enumerate}[label=\roman*.]
\item If $\varphi$ is proper lsc
$\gamma$-paraconvex with $\gamma>1$, its domain is an interval.
\item If $\varphi$ is $\Phi_{\gamma}$-convex with
$\Phi_{\gamma}=\left\{ \phi:\mathbb{R}\to\mathbb{R}:\phi\text{ is }\gamma\text{-paraconvex}\right\} $.
Then its domain is an interval.
\end{enumerate}
\end{lemma}

\begin{proof}
(i) is a straightforward derivation from the definition of paraconvex function i.e. for $x,y\in\mathrm{dom}\  \varphi$ we have 
\[
\left(\exists C>0,\forall\lambda\in\left[0,1\right]\right)\quad \varphi\left(\lambda x+\left(1-\lambda\right)y\right)\leq\lambda \varphi\left(x\right)+\left(1-\lambda\right)\varphi\left(y\right)+C\Vert x-y\Vert^{\gamma}<+\infty.
\]

(ii) Thanks to the property of abstract convexity, for every $\varepsilon>0,x\in\mathbb{R}$,
there exists $\varphi\geq\phi_{\varepsilon}\in\Phi_{\gamma}$ such that
\[
+\infty>\phi_{\varepsilon}\left(x\right)+\varepsilon> \varphi\left(x\right).
\]
Then for $z=\lambda x+\left(1-\lambda\right)y$ with $x,y\in\mathrm{dom}\ \varphi$,
for every $\varepsilon>0$, there exists $\phi_{\varepsilon}\leq \varphi$
such that 
\begin{align*}
\varphi\left(z\right) & <\phi_{\varepsilon}\left(z\right)+\varepsilon\leq\lambda\phi_{\varepsilon}\left(x\right)+\left(1-\lambda\right)\phi_{\varepsilon}\left(y\right)+C_{\phi_{\varepsilon}}\Vert x-y\Vert^{\gamma}+\varepsilon\\
 & \leq\lambda \varphi\left(x\right)+\left(1-\lambda\right)\varphi\left(y\right)+C_{\phi_{\varepsilon}}\Vert x-y\Vert^{\gamma}+\varepsilon<+\infty.
\end{align*}
\end{proof}
\begin{remark}
It seems to us that weakly convex implies $\Phi_{lsc}^\mathbb{R}$-convex. This cannot be generalized to $\gamma$-paraconvex and $\Phi_\gamma$-convex functions. Hence, under which condition can this be true? One answer would be if $\varphi +\rho\Vert\Vert^\gamma$ is convex, then $\varphi$ is $\Phi_\gamma$-convex.  
\end{remark}

To prove that $\gamma$-subdifferentials is bounded, the following lemma will be useful.

\begin{lemma}
\label{lem: paraconv int epi nonempty}
Consider $\varphi:H\to(-\infty,+\infty]$ be proper $\gamma$-paraconvex and $x_{0}\in O \subset\mathrm{dom}\ \varphi$ where $O$ is an open convex set. Then $\varphi$ is bounded above on a neighborhood of $x_0$ and the interior of $\mathrm{epi}\ \varphi$ is non-empty.
\end{lemma}
\begin{proof}
The first claim comes from Lipschitz continuity of $\varphi$ around $x_0$ which leads to $\varphi$ be bounded above around $x_0$ (see \cite[Proposition 2.2]{jourani1996subdifferentiability}).
The proof of the second claim follows the lines of the proof of \cite[Proposition 8.45]{Bau2011}.
\end{proof}
We present the result for paraconvex functions.
\begin{proposition}
\label{prop: paraconv subgrad bounded}
Let $\varphi:H\to(-\infty,+\infty]$ be proper $\gamma$-paraconvex with $1<\gamma\leq2$ and $x_{0}\in O \subset\mathrm{dom}\ \varphi$ where $O$ is an open convex set. 
Assume 
$N_{\gamma}\left(x_0,\varphi\left(x_0\right);\mathrm{epi}\  \varphi\right)\backslash\left\{ \left(0,0\right)\right\} \neq \emptyset$.
Then $\partial_{\gamma}\varphi\left(x_{0}\right)$ is non-empty and weakly compact. Moreover, there exists $\rho>0$ such that $\partial_{\gamma} \varphi (B(x_0,\rho))$ is bounded.
\end{proposition}
\begin{proof}
Observe that $\left(x_{0},\varphi \left(x_{0}\right)\right)$ belongs to the boundary of $\mathrm{epi}\ \varphi$ as $(x_0,\varphi(x_0)-\varepsilon)\notin \mathrm{epi }\ \varphi$ for $\varepsilon>0$. 
By assumption, let us take $\left(u,v\right)\in N_{\gamma}\left(x_0,\varphi\left(x_0\right);\mathrm{epi}\ \varphi\right)\backslash\left(0,0\right)$.
For every $y\in\mathrm{dom}\ \varphi,\eta\geq0$, we have $\left(x_0,\varphi\left(y_0\right)+\eta\right)\in\mathrm{epi}\ \varphi$ and hence, 
\begin{equation}
v\eta \leq C\eta^\gamma.
\label{eq: subgrad normal cone inequality}
\end{equation}
Letting $\eta\to 0$, we obtain that $v$ musts be non-positive. If $v=0$ then $u$ musts be zero due to Lipschitz continuity of $\varphi$ around $x_0$.
This contradicts $\left(u,v\right)\neq\left(0,0\right)$. Hence $v<0$, we obtain that 
\[
\left(\frac{u}{\left|v\right|},-1\right)=\frac{1}{\left|v\right|}\left(u,v\right)\in N_{\gamma}\left(x_0,\varphi\left(x_0\right);\mathrm{epi}\ \varphi\right),
\]
which implies $u/\left|v\right|\in\partial_{\gamma, \mathrm{loc}} \varphi\left(x_0\right)\ne\emptyset$ which is the local $\gamma$-subdifferentials of $\varphi$ at $x_0$.
Then by \cite[Proposition 3.1]{jourani1996subdifferentiability}, $\gamma$-paraconvex functions have globalization property which implies the global $\gamma$-subdifferentials $\partial_{\gamma} \varphi(x_0)$ is non-empty.

On the other hand, we also know that $\varphi$ is Lipschitz continuous in the neighborhood of $x_0$. There exist $\rho,\beta>0$ such that $\left|\varphi\left(y\right)-\varphi\left(z\right)\right|\leq\beta\left\Vert y-z\right\Vert $ for all $y,z\in B\left(x_0,2\rho\right)$. 
Take $y\in B\left(x_0,\rho\right),w_y\in\partial_{\gamma} \varphi\left(y\right)$. By the definition of $\gamma$-subdifferential at $\varphi(y)$, we have 
\[
(\exists C>0, \forall z\in H) \quad \varphi(z)- \varphi(y) \geq \langle w_y,z-y\rangle -C\Vert z-y\Vert^\gamma.
\]
Then for all $z\in B\left(0,\rho\right), y\in B(x_0,\rho)$, we obtain $y,y+z\in B(x_0,2\rho)$. Then, 
\begin{align*}
\left\langle w_y,z\right\rangle & \leq \varphi\left(y+z\right)-\varphi\left(y\right)+C\left\Vert z\right\Vert ^{\gamma}\\
&\leq\beta\left\Vert z\right\Vert +C\left\Vert z\right\Vert ^{\gamma} \leq \beta +C\rho^{\gamma-1}.
\end{align*}
We get $\left\Vert w_y\right\Vert \leq\beta +C\rho^{\gamma-1}$. 
Then $\partial_{\gamma} \varphi\left(y\right)\subset\partial_{\gamma, \mathrm{loc}} \varphi\left(B\left(x_0,\rho\right)\right)\subset B\left(0,\beta \rho+C\rho^{\gamma}\right)$ for all $y\in B(x_0,\rho)$.
Thus, $\partial_{\gamma} \varphi\left(x_0\right)$ is bounded.
Moreover, $\partial_{\gamma} \varphi\left(x_0\right)$ is
also closed and convex so it is weakly compact.
\end{proof}

\subsection{Continuity of Paraconvex functions on the real line}

We state an auxiliary result for paraconvex functions which is crucial for the proof of Lemma \ref{lem: paraconv subdiff real line}.
\begin{proposition}[Three points paraconvexity]
\label{prop:3-pts paraconvex}
Let $f:\mathbb{R}\to (-\infty,+\infty]$ be $\gamma$-paraconvex with $\gamma>1$, then for any $x<y<z$, there exists a constant $C>0$ such that
\begin{equation}
\frac{f\left(y\right)-f\left(x\right)}{y-x}-C\left(z-y\right)\left(z-x\right)^{\gamma-2}\leq\frac{f\left(z\right)-f\left(x\right)}{z-x}\leq\frac{f\left(z\right)-f\left(y\right)}{z-y}+C\left(y-x\right)\left(z-x\right)^{\gamma-2}.
\end{equation}
\end{proposition}
\begin{proof}
Let us set $\lambda = \frac{z-y}{z-x}\in (0,1)$, so that $y=\lambda x +(1-\lambda)z$. By the definition of paraconvexity, there exists $C>0$ such that
\begin{align}
    f(y) & \leq \lambda f(x)+(1-\lambda) f(z) +C\lambda (1-\lambda) \Vert x-z\Vert^\gamma \nonumber\\
    &= \frac{z-y}{z-x} f(x) +\frac{y-x}{z-x} f(z) +C(z-y)(y-x) (z-x)^{\gamma-2}.\nonumber
\end{align}
Multiplying both sides by $(z-x)$, we obtain
\begin{equation}
\label{eq: para-convex}
    (z-x) f(y) \leq (z-y) f(x) +(y-x) f(z) +C(z-y)(y-x) (z-x)^{\gamma-1}.
\end{equation}
From \eqref{eq: para-convex}, grouping $f(y) -f(x)$, we have
\begin{equation}
f(y)-f(x) \leq \frac{y-x}{z-x} (f(z)-f(x)) +C(z-y)(y-x) (z-x)^{\gamma-2}. \nonumber
\end{equation}
Divide both sides by $(y-x)$, we obtain the first inequality. The second inequality can be proved by grouping $f(z)-f(x)$.
\end{proof}
We exploit Proposition \ref{prop:3-pts paraconvex} to prove that paraconvex functions have the same property as convex functions for upper and lower limit at the end point \cite[Proposition 3.1.2]{hiriart2013convex}.
\begin{proposition}
\label{prop: paraconvex-cont end point}
Let $f:\mathbb{R}\to (-\infty,+\infty]$ be $\gamma$-paraconvex with $\gamma>1$ and the domain of $f$ has non-empty interior with $a\in \mathbb{R}$ is the left-end point. Then the right-limit $f(a_+):= \lim_{x\downarrow a} f(x)$ exists and $f(a)\geq f(a_+)$.

Similarly, for the right-end point $b\in\mathbb{R}$, $f(b_-)$ exists and $f(b) \geq f(b_-)$.
\end{proposition}
\begin{proof}
Let $x_0\in \mathrm{int}\ \mathrm{dom}\ f$ and $t_0=x_0-a>0$. Let us set
\[
q(t) := \frac{f(x_0-t) -f(x_0)}{t},
\]
for $t>0$ and $l:= \lim_{t\uparrow t_0} q(t)$. We will prove the existence of $l$ at the end.
Notice that $t\uparrow t_0$ then $x_0-t \downarrow a$. As $f$ is $\gamma$-paraconvex, it is locally Lipschitz continuous on $\mathrm{int} \mathrm{dom} f$ and so is continuous on the interior of the domain. Then as $t\uparrow t_0$, we have 
\[
f(x_0 -t) = f(x_0) +q(t)t \to f(x_0) + (x_0-a) l = f(a_+)\in (-\infty,+\infty].
\]
Now we use Proposition \ref{prop:3-pts paraconvex} with $z=x_0,y=x_0-t, x=a$ to obtain 
\[
\frac{f\left(x_0\right)-f\left(a\right)}{x_0 -a}\leq\frac{f\left(x_0\right)-f\left(x_0-t\right)}{t}+C\left(x_0 -t -a\right)\left(x_0 -a\right)^{\gamma-2}.
\]
Let $t\uparrow t_0$, we obtain
\[
\frac{f\left(x_0\right)-f\left(a\right)}{x_0 -a}\leq\frac{f\left(x_0\right)-f\left(a_+\right)}{x_0-a},
\]
which implies $f(a_+)\leq f(a)$ so the limit $f(a_+)$ exists and finite. 

Now let us prove the existence of the left side limit of $f$ at $a$. Depending on the sign of $a$, we construct a monotone operator so which guarantees one-sided limit exists.

Generally, let us consider $x<y<z$ and divide into two cases. 
If $y\geq 0$, set $q\left(x\right)=\frac{f\left(z\right)-f\left(x\right)}{z-x}$, using Proposition \ref{prop:3-pts paraconvex} second inequality, we obtain
\[
q\left(x\right)\leq q\left(y\right)+C\left(y-x\right)\left(z-x\right)^{\gamma-2}.
\]
As $z-x\geq z-y$, we bound the above inequality with
\[
q\left(x\right)+Cx\left(z-x\right)^{\gamma-2}\leq q\left(y\right)+Cy\left(z-y\right)^{\gamma-2},
\]
so the function $h\left(y\right)=q\left(y\right)+Cy\left(z-y\right)^{\gamma-2}$ is monotone with fixed $z$.

If $y<0$, we use the first inequality of Proposition \ref{prop:3-pts paraconvex}, set $q\left(y\right)=\frac{f\left(y\right)-f\left(x\right)}{y-x}$ and obtain  
\[
q\left(y\right)+Cy\left(y-x\right)^{\gamma-2} \leq q\left(z\right)+Cz\left(z-x\right)^{\gamma-2},
\]
which implies the monotonicity of $h(y) = q\left(y\right)+Cy\left(y-x\right)^{\gamma-2}$ with fixed $x$. Hence, for $a\geq 0$, the first case applied and when $a<0$, we can take $x=a<y<0$ and use the second case. For right-end point $b$, the same argument can be applied.
\end{proof}
\section{Conclusions}
In Theorem \ref{thm: prox relation general}, we provide a characterization of a mapping $f$, which is a proximity operator in the context of abstract convex functions.
Specifically, Corollary \ref{cor: lsc prox relation} gives us the characterization of the proximity operator in the sense of the definition provided in \cite[Formula 22]{bednarczuk2025proximal} (see also Example \ref{exm1} of the present paper). 
This definition refers to the class $\Phi_{lsc}^\mathbb{R}$-convex functions. Some examples of the functions $D(\cdot,\cdot)$ from \cite{gribonval2020characterization} can also be applied in the context of Theorem \ref{thm: prox relation general} in the present paper.

\paragraph{Disclosure Statement:} The authors report there are no competing interests to declare.

\bibliographystyle{plain}
\bibliography{references}

@BOOK{Bau2011,
author = {Bauschke, Heinz H. and Combettes, Patrick L.},
title = {Convex Analysis and Monotone Operator Theory in {H}ilbert Spaces},
year = {2017},
isbn = {978-3-319-48310-8},
publisher = {Springer Publishing Company, Incorporated},
edition = {2nd}}

@BOOK{Rub2013,
  author    = {Rubinov, Alexander M},
  publisher = {Springer Science \& Business Media},
  title     = {Abstract convexity and global optimization},
  year      = {2013},
  volume    = {44},
}

@BOOK{Pall2013,
  title={Foundations of mathematical optimization: convex analysis without linearity},
  author={Pallaschke, Diethard Ernst and Rolewicz, Stefan},
  volume={388},
  year={2013},
  publisher={Springer Science \& Business Media}
}

@ARTICLE{andra2002,
  title={A survey of methods of abstract convex programming},
  author={Andramonov, Mikhail},
  journal={Journal of Statistics and Management Systems},
  volume={5},
  number={1-3},
  pages={21--37},
  year={2002},
  publisher={Taylor \& Francis}
}

@ARTICLE{moreau1970,
  title={Inf-convolution, sous-additivit{\'e}, convexit{\'e} des fonctions num{\'e}riques},
  author={Moreau, Jean Jacques},
  journal={Journal de Math{\'e}matiques Pures et Appliqu{\'e}es},
  year={1970}
}

@article{davis2019,
  title={Stochastic model-based minimization of weakly convex functions},
  author={Davis, Damek and Drusvyatskiy, Dmitriy},
  journal={SIAM Journal on Optimization},
  volume={29},
  number={1},
  pages={207--239},
  year={2019},
  publisher={SIAM}
}

@article{jourani1996subdifferentiability,
  title={Subdifferentiability and subdifferential monotonicity of 1-paraconvex functions},
  author={Jourani, Abderrahim},
  journal={Control and Cybernetics},
  volume={25},
  number={4},
  year={1996},
  pages={721--737}
}

@article{syga2019global,
  title={On global properties of lower semicontinuous quadratically minorized functions},
  author={Syga, Monika},
  journal={arXiv preprint arXiv:1912.04644},
  year={2019}
}

@article{moreau1965proximite,
  title={Proximit{\'e} et dualit{\'e} dans un espace hilbertien},
  author={Moreau, Jean-Jacques},
  journal={Bulletin de la Soci{\'e}t{\'e} math{\'e}matique de France},
  volume={93},
  pages={273--299},
  year={1965}
}

@article{kiwiel1997proximal,
  title={Proximal minimization methods with generalized Bregman functions},
  author={Kiwiel, Krzysztof C},
  journal={SIAM journal on control and optimization},
  volume={35},
  number={4},
  pages={1142--1168},
  year={1997},
  publisher={SIAM}
}

@article{rolewicz1979gamma,
  title={On $\gamma$-paraconvex multifunctions},
  author={Rolewicz, Stefan},
  journal={Math. Japonica},
  volume={24},
  pages={293--300},
  year={1979}
}

@book{hiriart2013convex,
  title={Convex analysis and minimization algorithms II: Advanced Theory and Bundle Methods},
  author={Hiriart-Urruty, Jean-Baptiste and Lemar{\'e}chal, Claude},
  volume={306},
  year={2013},
  publisher={Springer science \& business media}
}

@book{clarke2008nonsmooth,
  title={Nonsmooth analysis and control theory},
  author={Clarke, Francis H and Ledyaev, Yuri S and Stern, Ronald J and Wolenski, Peter R},
  volume={178},
  year={2008},
  publisher={Springer Science \& Business Media}
}

@article{chambolle2011first,
  title={A first-order primal-dual algorithm for convex problems with applications to imaging},
  author={Chambolle, Antonin and Pock, Thomas},
  journal={Journal of mathematical imaging and vision},
  volume={40},
  pages={120--145},
  year={2011},
  publisher={Springer}
}

@article{millan2024global,
  title={Global minimisation of nonconvex functions by generalising the mirror descent method},
  author={Mill{\'a}n, Reinier D{\'\i}az and Ugon, Julien},
  journal={arXiv preprint arXiv:2402.04281},
  year={2024}
}

@article{bednarczuk2025proximal,
  title={Proximal Algorithms for a class of abstract convex functions},
  author={Bednarczuk, Ewa and Lorenz, Dirk and Tran, The Hung},
  journal={Set-Valued and Variational Analysis},
  volume={33},
  number={1},
  pages={5},
  year={2025},
  publisher={Springer}
}

@article{hoheisel2010proximal,
  title={On proximal point-type algorithms for weakly convex functions and their connection to the backward euler method},
  author={Hoheisel, Tim and Laborde, Maxime and Oberman, Adam},
  journal={Optimization Online},
  year={2010}
}

@article{jourani2014differential,
  title={Differential properties of the Moreau envelope},
  author={Jourani, Abderrahim and Thibault, Lionel and Zagrodny, Dariusz},
  journal={Journal of Functional Analysis},
  volume={266},
  number={3},
  pages={1185--1237},
  year={2014},
  publisher={Elsevier}
}

@book{berge1877topological,
  title={{Topological spaces: Including a treatment of multi-valued functions, vector spaces and convexity}},
  author={Berge, Claude},
  year={1877},
  publisher={Oliver \& Boyd},
  address={Edinburgh}
}

@article{gribonval2020characterization,
  title={A characterization of proximity operators},
  author={Gribonval, R{\'e}mi and Nikolova, Mila},
  journal={Journal of Mathematical Imaging and Vision},
  volume={62},
  number={6},
  pages={773--789},
  year={2020},
  publisher={Springer}
}

@article{rahimi2024projected,
  title={{Projected subgradient methods for paraconvex optimization: Application to robust low-rank matrix recovery}},
  author={Rahimi, Morteza and Ghaderi, Susan and Moreau, Yves and Ahookhosh, Masoud},
  journal={arXiv preprint arXiv:2501.00427},
  year={2024}
}

@book{evans2025measure,
  title={Measure theory and fine properties of functions},
  author={Evans, Lawrence C},
  year={2025},
  publisher={Chapman and Hall/CRC}
}

@article{rakotomandimby2026subgradient,
  title={{Subgradient selector in the generalized cutting plane method with an application to sparse optimization: S. Rakotomandimby et al.}},
  author={Rakotomandimby, Seta and Chancelier, Jean-Philippe and De Lara, Michel and Le Franc, Adrien},
  journal={Optimization Letters},
  volume={20},
  number={2},
  pages={473--490},
  year={2026},
  publisher={Springer}
}

@article{bauschke2018regularizing,
  title={Regularizing with Bregman--Moreau envelopes},
  author={Bauschke, Heinz H and Dao, Minh N and Lindstrom, Scott B},
  journal={SIAM Journal on Optimization},
  volume={28},
  number={4},
  pages={3208--3228},
  year={2018},
  publisher={SIAM}
}

@article{kabgani2025moreau,
  title={Moreau envelope and proximal-point methods under the lens of high-order regularization},
  author={Kabgani, Alireza and Ahookhosh, Masoud},
  journal={Set-Valued and Variational Analysis},
  volume={33},
  number={4},
  pages={47},
  year={2025},
  publisher={Springer}
}

@article{laude2023anisotropic,
  title={Anisotropic proximal point algorithm},
  author={Laude, Emanuel and Patrinos, Panagiotis},
  journal={arXiv preprint arXiv:2312.09834},
  year={2023}
}

@article{bednarczuk2026outer,
  title={Outer approximation scheme for weakly convex constrained optimization problems},
  author={Bednarczuk, Ewa M and Bruccola, Giovanni and Pesquet, Jean-Christophe and Rutkowski, Krzysztof},
  journal={Journal of Global Optimization},
  volume={94},
  number={4},
  pages={1137--1166},
  year={2026},
  publisher={Springer}
}

@article{rolewicz2000cyclic,
  title={On cyclic $\alpha$ ({\textperiodcentered})-monotone multifunctions},
  author={Rolewicz, S},
  journal={Studia Mathematica},
  volume={141},
  number={3},
  pages={263--272},
  year={2000},
  publisher={Polska Akademia Nauk. Instytut Matematyczny PAN}
}

@incollection{giles1999survey,
  title={A survey of Clarke’s subdifferential and the differentiability of locally Lipschitz functions},
  author={Giles, JR},
  booktitle={Progress in Optimization: Contributions from Australasia},
  pages={3--26},
  year={1999},
  publisher={Springer}
}
\end{document}